\documentclass[11pt]{amsart}

\usepackage[utf8]{inputenc}
\usepackage{graphicx}
\usepackage{xcolor}
\usepackage{amsfonts}
\usepackage{amsmath}
\usepackage{amsthm}
\usepackage{amssymb}
\usepackage{pgfplots}
\pgfplotsset{compat=1.17}
\usepackage{graphicx}
\usepackage{float}
\usepackage{upgreek}
\usepackage{siunitx}
\usepackage{booktabs}
\usepackage{textcomp, gensymb}
\usepackage{enumitem}
\usepackage{multicol}
\usepackage{mathrsfs}
\usepackage{algpseudocode}
\usepackage{adjustbox}
\usepackage{algorithm}
\usepackage{mathtools}
\usepackage{tikz-cd}
\usepackage[font=small,labelfont=bf]{caption}
\usepackage{wrapfig}
\usepackage[margin=1.24in]{geometry}
\usepackage{datetime}
\usepackage[
    backend=biber,
    style=ieee,
    citestyle=numeric-comp,
    sorting=nyt,
    dashed=false
]{biblatex}
\usepackage{hyperref}
\usepackage{xurl}
\hypersetup{
    colorlinks=true,
    linkcolor=blue,
    citecolor=red,
    breaklinks=true
}

\newdateformat{monthyeardate}{\monthname[\THEMONTH] \THEYEAR}

\title[Quantum harmonic analysis on locally compact groups]{Quantum harmonic analysis on locally\\ compact groups}
\author{Simon Halvdansson}
\address{Department of Mathematical Sciences, Norwegian University of Science and Technology, 7491 Trondheim, Norway.}
\email{simon.halvdansson@ntnu.no}

\date{\monthyeardate\today}

\theoremstyle{plain}
\newtheorem{theorem}{Theorem}[section]
\newtheorem*{theorem*}{Theorem}
\newtheorem{lemma}[theorem]{Lemma}
\newtheorem{proposition}[theorem]{Proposition}
\newtheorem{corollary}[theorem]{Corollary}

\theoremstyle{definition}
\newtheorem{definition}[theorem]{Definition}

\newtheorem{example}[theorem]{Example}

\theoremstyle{remark}
\newtheorem*{remark}{Remark}

\newcommand{\C}{\mathbb{C}}
\newcommand{\R}{\mathbb{R}}

\newcommand{\tr}{\operatorname{tr}}
\newcommand{\aff}{\textrm{Aff}}
\newcommand{\D}{\mathcal{D}}
\newcommand{\dmr}{\, d\mu_r}
\newcommand{\dml}{\, d\mu_\ell}

\makeatletter
\newcommand{\vast}{\bBigg@{4}}
\newcommand{\Vast}{\bBigg@{5}}
\makeatother

\def\XXint#1#2#3{{\setbox0=\hbox{$#1{#2#3}{\int}$ }
		\vcenter{\hbox{$#2#3$ }}\kern-.6\wd0}}

\begin{document}
	\maketitle
	\begin{abstract}\vspace{-9mm}
	    On a locally compact group we introduce covariant quantization schemes and analogs of phase space representations as well as mixed-state localization operators. These generalize corresponding notions for the affine group and the Heisenberg group. The approach is based on associating to a square integrable representation of the locally compact group two types of convolutions between integrable functions and trace class operators. In the case of non-unimodular groups these convolutions only are well-defined for admissible operators, which is an extension of the notion of admissible wavelets as has been pointed out recently in the case of the affine group. 
	    \vspace{6mm}
	\end{abstract}
	
	\renewcommand{\thefootnote}{\fnsymbol{footnote}}
    \renewcommand{\thefootnote}{\arabic{footnote}}
	
	\section{Introduction}
	The theory of time-frequency representations, aka phase space representations, and pseudo-differential operators has been extended from the Euclidean framework to various other settings by replacing the Schr\"odinger representation of the Weyl-Heisenberg group by a unitary representation of a locally compact group, see  \cite{kumar2021trace,mantoiu2017pseudo} for recent contributions to this circle of ideas. 
	
	A particular class of pseudodifferntial operators has been intensively studied in mathematics, quantum mechanics and time-frequency analysis which is known as Toeplitz operators, localization operators and anti-Wick quantization or covariant integral quantization in the respective fields, see \cite{Wong2002, Wong2002_Lp, Li2014,Gazeau2016, Gazeau2020,Kiukas2006} for contributions to the non-Euclidean setting.
	
	In \cite{Luef2018} the authors have established a link between the theory of localization operators and  
	quantum harmonic analysis on phase space, the latter had been introduced by Werner \cite{Werner1984}. In \cite{Werner1984} a convolution $f\star S$ between a Lebesgue integrable function $f$ and trace class operator $S$, and a convolution, $T \star S$, of two trace class operators $T, S$ are defined and shown to behave in a manner analogous to the convolution of two functions. 
	
	In a series of papers Luef and Skrettingland  have demonstrated the merits of viewing localization operators as the convolution of a function and a rank-one operator \cite{Luef2019_acc, Luef2019, Luef2021}, phase retrieval \cite{Luef2019}, frame theory \cite{skrettingland2020_gabor} and convolutional neural networks \cite{Dorfler2021}. 
	
   Furthermore, it has been noticed in \cite{Luef2019} that the time-frequency representations associated to the generalization of localization operators, $f\star S$, are Cohen class distributions defined in terms of $S$. The formulation of statements in time-frequency analysis in terms of Werner's convolution has turned out to be very fruitful and has been extended to the affine group in \cite{Berge2022}.
	
	Meanwhile, there has been interest in related problems in the more general setting of square integrable representations of locally compact groups. Examples of this include coorbit spaces \cite{Feichtinger1988, FEICHTINGER1989307, Feichtinger1989, Berge2021_coorbit, Balazs2019, Romero2012}, localization operators \cite{Wong2002, Wong2002_Lp, Li2014}, covariant integral quantizations \cite{Gazeau2016, Gazeau2020, Kiukas2006}, reproducing kernel Hilbert spaces \cite{Berge2021_rigid} and sample reconstruction \cite{Fhr2006}. Consequently, it is the goal of this paper to set up the theory of quantum harmonic analysis on locally compact groups, inspired by the construction for the affine group in \cite{Berge2022} and use it to establish notions and theorems from time-frequency analysis and time-scale analysis in the more general case of locally compact groups .
	
	The main objects in quantum harmonic analysis are the function-operator and operator-operator convolutions, which we define in this paper for a locally compact group $G$,  and a square integrable unitary representation $\sigma$ on a Hilbert space $\mathcal{H}$ as
	$$
	f \star S = \int_G f(x) \sigma(x)^* S \sigma(x)\dmr(x),\qquad T \star S (x) = \tr(T \sigma(x)^* S \sigma(x)),
	$$
	where $\mu_r$ denotes the right-invariant Haar measure on $G$.
	
	These definitions are motivated by the idea that the mapping $S \mapsto \sigma(x)^* S \sigma(x)$ corresponds to a \emph{translation} of an operator and that the trace measures the size of an operator in the same way as the integral measures the size of a function. It turns out that many properties of convolutions such as associativity and a version of Young's inequality hold mutantis mutadis for this type of convolutions, too. 
	
	The primary motivation for studying these convolutions is that $f \star S$ is a localization operator on $G$  and $S \star T$ a Cohen's class distribution if $S$ is a rank-one operator. The general case is also of interest and can be seen to correspond to extensions of the short-time Fourier transform with respect to an operator window introduced in \cite{skrettingland2020equivalent_norms} to general wavelet transforms on a locally compact group $G$.
	
	Studying the issue of integrability of operator-operator convolutions leads one to a definition of \emph{admissibility of operators}, generalizing the concept of admissible vectors for the wavelet transform. This criterion turns out to be important to applications and is the main source of discrepancy from the corresponding theory for the Weyl-Heisenberg case where all trace-class operators are admissible. 
	
	Most of the novel contributions of this paper are contained in Section \ref{sec:applications} including an uncertainty principle for Cohen's class distributions, results on the distribution of eigenvalues of mixed-state localization operators for the wavelet transform, Berezin-Lieb inequalities for both function-operator and operator-operator convolutions and a version of Wiener's Tauberian theorem giving equivalent conditions for translates of an operator to be dense in the Schatten classes $\mathcal{S}^p$ for $1\le p\le\infty$. Hence, we are able to establish a theory of quantum harmonic analysis on a locally compact group $G$, which has just one deficiency compared to \cite{Werner1984}; the lack of a multiplication theorem for the operator-valued Fourier transform.

	\subsubsection*{Outline} In Section \ref{sec:preliminaries}, we go over some preliminaries on operator theory, time-frequency analysis and locally compact groups without discussing quantum harmonic analysis in too much depth. This section can be skipped over if the reader is familiar with works such as \cite{Berge2022, Luef2018, Luef2019, Luef2019_acc}. Notably, Section \ref{sec:motivating_examples} lists three examples of square integrable representations which motivate the generalizations in the paper. In Section \ref{sec:convolutions}, we define the three convolutions in quantum harmonic analysis; function-function, function-operator and operator-operator, and establish some elementary properties. Section \ref{sec:admissibility} is devoted to the construction and properties of admissible operators while Section \ref{sec:applications} goes through applications mainly related to the notion of admissible operators.
	
	\subsubsection*{Notational conventions}
	Throughout this article, a general locally compact group will be denoted by $G$ with the zero element denoted by $0_G$, general elements denoted by $x, y, z$ and the associated left and right Haar measures written as $\mu_\ell$ and $\mu_r$, respectively. Furthermore,  $\mathcal{H}$ will denote a Hilbert space, any norm without a subscript will be assumed to be taken in $\mathcal{H}$ and for an operator $A$, $A^*$ will denote its adjoint. The set $\mathcal{U}(\mathcal{H})$ is the set of all unitary operators on $\mathcal{H}$ and the most notable members are unitary square integrable representations $\sigma$ of $G$. For $p < \infty,\, \mathcal{S}^p$ denotes the Schatten $p$-class of operators with singular values in $\ell^p$ and by $\mathcal{S}^\infty$ we mean $B(\mathcal{H})$, the set of all bounded linear operators on $\mathcal{H}$. 
	
	\section{Preliminaries}\label{sec:preliminaries}
	In this section we go over some of the preliminaries of quantum harmonic analysis, time-frequency analysis and locally compact groups. The exposition is similar to that in \cite{Luef2018, Luef2019_acc, Berge2022} and related works. See also \cite{Wong2002} for an introduction more focused on the setting of locally compact groups.
	
	\subsection{Operator theory}
	\subsubsection{Singular value decomposition}
	We will frequently make use of the singular value decomposition of a compact operator $A$ which has the form
	$$
	A = \sum_n s_n(A) (\psi_n \otimes \phi_n)
	$$
	where $\{\psi_n\}_n$ and $\{ \phi_n \}_n$ are two orthonormal sets in $\mathcal{H}$, $(s_n(A))_n$ is a sequence converging to zero and $\psi \otimes \phi$ is the \emph{rank-one operator}, defined by $(\psi \otimes \phi) (\xi) = \langle \xi, \phi\rangle \psi$. The sum converges in the strong topology of $B(\mathcal{H})$ and the numbers $s_n(A)$ are the eigenvalues of $\sqrt{A^*A}$ and are called the singular values of $A$.

In case $A$ is a positive, compact operator, then we can take $\phi_n = \psi_n$ and the singular values $s_n(A)$ agree with  the eigenvalues $\lambda_n(A)$ of $A$. Hence, we have the following spectral decomposition of $A$:
	$$
	A = \sum_n \lambda_n(A) (\psi_n \otimes \psi_n).
	$$
	\subsubsection{Schatten classes of operators}
	The Schatten class $\mathcal{S}^p$ is the space of compact operators with singular values in $\ell^p$ for $p\in[1,\infty]$. In particular, the space $\mathcal{S}^1$ is referred to as the space of \emph{trace-class} operators and $\mathcal{S}^2$ as the space of \emph{Hilbert-Schmidt} operators. It is a non-trivial fact that $\mathcal{S}^p$ is a Banach space for any $1 \leq p \leq \infty$ and that $\mathcal{S}^2$ is a Hilbert space with the inner product $\langle S, T \rangle_{\mathcal{S}^2} =\tr(ST^*)$. For a trace-class operator $S \in \mathcal{S}^1$, we define the \emph{trace} of $S$ as
	$$
	\tr(S) = \sum_n \langle Se_n, e_n \rangle
	$$
	where $\{e_n\}_n$ is an orthonormal basis. This quantity is finite and independent of the chosen orthornomal basis. Similarly, it turns out that the Schatten p-norm of $S \in \mathcal{S}^p$ can be shown to be equal to $\Vert S \Vert_{\mathcal{S}^p}^p = \tr(|S|^p)$ for $1 \leq p < \infty$ where $|S|$ is the absolute value of $S$.
	
	In the same way as for $L^p$-spaces of functions, we define duality brackets for conjugate $\mathcal{S}^p$ spaces as
	$$
	\langle A, B \rangle_{\mathcal{S}^p, \mathcal{S}^q} = \tr(AB)
	$$
	where $A \in \mathcal{S}^p, B \in \mathcal{S}^q$ and $\frac{1}{p} + \frac{1}{q} = 1$.
	\subsubsection{Vector valued integration}\label{sec:vector_integration}
	In defining convolutions between functions and operators, we will need to integrate operator valued functions $H : G \to B(\mathcal{H})$ which are of the form $H(x) = f(x) F(x)$ where $f \in L^1_r(G)$ and $F : G \to B(\mathcal{H})$ are measurable, bounded and strongly continuous. The operator-valued integral of $H$ is then defined weakly as
	\begin{align*}
	    \left\langle \left( \int_G f(x) F(x) \,d\mu \right) \psi, \phi\right\rangle = \int_G f(x) \langle F(x)\psi, \phi \rangle \,d\mu
	\end{align*}
	for $\psi, \phi \in \mathcal{H}$ where $\mu$ is a measure on $G$. For more on this matter of defining operators via integrals, known as Bochner integration, see the discussion in \cite[Sec. 2.3]{Luef2018}.

    \subsection{Time-frequency analysis}
    We briefly introduce some of the main objects of time-frequency analysis. For a more thorough introduction, see e.g. \cite{grochenig_book, daubechiesTen}.
    
    \subsubsection{Short-time Fourier transform}
    Perhaps the most classical tool of time-frequency analysis is the following \emph{time-frequency representation}: Given $\psi, \varphi \in L^2(\R^d)$, the \emph{short-time Fourier transform} (STFT) of $\psi$ with respect to the \emph{window} $\varphi$ is the function
    $$
    V_\varphi \psi(x, \omega) = \int_{\R^d}\psi(t) \overline{\varphi(t-x)} e^{-2\pi i \omega\cdot t}\,dt
    $$
    on $\R^{2d}$. It has shown to be useful to consider the STFT to be induced by the \emph{(projective) representation} $\pi(x, \omega) = M_\omega T_x$ of the \emph{Weyl-Heisenberg group} where
    $$
    T_x f(t) = f(t-x),\qquad M_\omega f(t) = e^{2\pi i \omega t} f(t).
    $$
    In this notation we can write the STFT as $V_\varphi\psi(x,\omega) = \langle \psi, \pi(x,\omega) \varphi \rangle$. The STFT is a member of $L^2(\R^{2d})$ which can be seen by an application of \emph{Moyal's identity}
    $$
    \big\langle V_{\varphi_1}\psi_1, V_{\varphi_2}\psi_2\big\rangle_{L^2(\R^{2d})} = \langle \psi_1, \psi_2 \rangle \overline{\langle \varphi_1, \varphi_2 \rangle}.
    $$
    Often in applications the square of the modulus of the STFT, the \emph{spectrogram}, is used because it possesses many nice properties such as non-negativity. It is a \emph{quadratic} time-frequency representation.
    
    \subsubsection{Wigner distribution}\label{sec:wigner_prelim}
    Another quadratic time-frequency distribution is the \emph{Wigner distribution}, introduced by Wigner \cite{Wigner1932} in the 1930's. Given two functions $\psi, \phi \in L^2(\R^d)$, the cross-Wigner distribution of $\psi$ and $\phi$ is given by
    $$
    W(\psi, \phi)(x, \omega) = \int_{\R^d}\psi(t+x/2)\overline{\phi(t-x/2)}e^{-2\pi i \omega \cdot t}\,dt.
    $$
    When $\psi = \phi$, we simply write $W(\psi, \psi) = W(\psi)$ and call it the Wigner distribution.
    
    The Wigner distribution has numerous applications in engineering, mathematics and physics as a time-frequency representation. In addition, it is also the ``dual" object to Weyl quantization, a well-known \emph{quantization scheme} that  associates operators to functions, via
    $$
    \langle L_f \psi, \phi \rangle = \langle f, W(\phi, \psi) \rangle
    $$
    and the map $f \mapsto L_f$ is called the \emph{Weyl transform}, aka Weyl quantization.
    
    \subsubsection{Cohen's class of quadratic time-frequency distributions}\label{sec:cohen_prelims}
    Cohen's class provides a nice class of covariant quadratic time-frequency distributions, including the spectrogram, scalogram and Wigner distribution. It consists of all functions of the form
    $$
    Q_\Phi(\psi, \phi) = W(\psi, \phi) * \Phi
    $$
    where $\Phi$ is a function or tempered distribution. It turns out that many of the properties of Cohen's class distributions are determined by the Weyl transform of $\Phi$ and later in the paper, we will define Cohen's class of not just time-frequency distributions but distributions with different underlying groups using an operator as a replacement for the Weyl transform of $\Phi$. For more on Cohen's class of time-frequency distributions, see \cite{Cohen1989, Luef2019_acc}.
    
    \subsubsection{Localization operators}\label{sec:loc_op}
    Localization operators are classically defined with respect to the short-time Fourier transform as the operator valued integral
    \begin{align*}
        A_f^{\varphi_1, \varphi_2}(g) = \int_{\R^{2d}} f(x, \omega)V_{\varphi_1}g(x, \omega) \pi(x, \omega) \varphi_2\, dx\, d\omega.
    \end{align*}
    The function $f$ is referred to as the \emph{mask}, \emph{multiplier} or \emph{filter} and is often taken to be the indicator function of some compact set. In that case, the localization operator is written as $A_{\chi_\Omega}^{\varphi_1, \varphi_2} = A_\Omega^{\varphi_1, \varphi_2}$. Localization operators of the above form were originally introduced by I. Daubechies in \cite{Daubechies1990}.
    
    \subsection{Locally compact groups}
    Much of abstract harmonic analysis is carried out on locally compact groups because they possess many of the properties we need to define convolutions and other objects. For more on the general theory of harmonic analysis on locally compact groups, see \cite[Chap. 2]{Folland} and \cite{fuhr05}.
    
    \subsubsection{Left and right Haar measures}\label{sec:haar_measure}
    Given a locally compact group $G$, there always exists two Radon measures, the left Haar measure $\mu_\ell$ and the right Haar measure $\mu_r$. The left (right) Haar measure is left (right) invariant, meaning that $\mu_\ell(xE) = \mu_\ell(E)$ ($\mu_r(Ex) = \mu_r(E)$) for $x \in G$ and $E \subset G$. Both measures are equivalent in the sense that they are related via 
    $$
    \mu_r(E) = \mu_\ell(E^{-1}),\qquad \dmr(x) = \Delta_G(x^{-1})\dml(x),\qquad \dml(xy) = \Delta_G(y)\dml(x)
    $$
    where the function $\Delta_G : G \to (0, \infty)$ is called the \emph{modular function}. If $\Delta_G \equiv 1$, the group is said to be \emph{unimodular}. When discussing $L^p$-integrable functions $f : G \to \C$ with respect to the left and right Haar measures, we write
    $$
    L^p_\ell(G) = L^p(G, \dml),\qquad L^p_r(G) = L^p(G, \dmr)
    $$
    while for $p = \infty$, we simply write $L^\infty(G)$.
    
    \subsubsection{Weyl-Heisenberg group}
    The Weyl-Heisenberg group $\mathbb{H}^n = (\R^n \times \R^n \times \R, \cdot_{\mathbb{H}^n})$ is equipped with the group operation
    $$
    (x, \omega, t) \cdot_{\mathbb{H}^n} (x', \omega', t') = \left( x+x', \omega + \omega', t+t' + \frac{1}{2}(x'\omega - x\omega') \right).
    $$
    which should be compared with the the composition rule for time-frequency shifts
    \begin{align}\label{eq:WH_representation}
        (T_x M_\omega) (T_{x'} M_{\omega'}) = e^{2\pi i x' \cdot \omega} T_{x+x'}M_{\omega + \omega'}.
    \end{align}
    We are interested in the \emph{projective} representation
    $$
    \pi : \R^{2d} \to \mathcal{U}(L^2(\R^d)),\qquad \pi(x,\omega) = T_x M_\omega
    $$
    of the Weyl-Heisenberg group for which \eqref{eq:WH_representation} is the Mackey induced representation.

    \subsubsection{Wavelet transform}
    The wavelet transform is a \emph{time-scale representation} based on taking the inner product of a signal and translations and dilations of some window function. The dilation operator $D_a$ is defined for positive $a$ as
    $$
    D_a f(y) = \frac{1}{\sqrt{a}}f\left(\frac{y}{a}\right)
    $$
    and hence the wavelet transform has the form
    $$
    W_\phi \psi(x,a) = \big\langle \psi, T_x D_a \phi\big\rangle = \frac{1}{\sqrt{a}}\int_\R \psi(t) \overline{\phi \left(\frac{t-x}{a}\right)}\, dt.
    $$
    There exists a version of Moyal's identity for the wavelet transform, often referred to simply as the \emph{orthogonality relation}
    $$
    \langle V_{\phi_1}\psi_1, V_{\phi_2}\psi_2\rangle = \langle \psi_1, \psi_2 \rangle \langle \D^{-1}\phi_1, \D^{-1}\phi_2 \rangle
    $$
    where $\D$ denotes the Duflo-Moore operator of the affine group and is defined by
    $$
    \widehat{\D^{-1} \phi}(\omega) = \frac{\hat{\phi}(\omega)}{\sqrt{|\omega|}};
    $$
    and $\phi_1,\phi_2$ are two admissible wavelets. The square of the wavelet transform is referred to as the \emph{scalogram} in analogy to the spectrogram and used similarly.
    
    \subsubsection{Affine group}\label{sec:affine_group}
    The affine group $\aff = (\R \times \R^+, \cdot_\aff)$ has the group operation
    $$
    (x, a) \cdot_\aff (y, b) := (x+ay, ab)
    $$
    which coincides with the relation
    $$
    (T_x D_a)(T_y D_b) = T_x T_{ay} D_a D_b = T_{x+ay}D_{ab}
    $$
    between translation and dilation operators. Thus the representation $\pi(x,a) = T_x D_a$ of the affine group induces the wavelet transform discussed above. Moreover, it is easy to see that $\pi(x,a) = T_x D_a$ is a unitary representation of $\aff$ on $L^2(\R)$. Often, when dealing with wavelet analysis we are only interested in \emph{analytic signals} which are $L^2$-functions for which $\hat{f}(\omega) = 0$ for $\omega < 0$ because the above representation is irreducible on the Hardy space of the real line. Using the Plancherel theorem, we can do everything on the Fourier side where our underlying Hilbert space becomes $\mathcal{H} = L^2(\R^+)$ which will be the case for the remainder of this paper when discussing quantum harmonic analysis on the affine group.
    
    A quick calculation shows that the inverse $(x,a)^{-1}$ of an element $(x,a) \in \aff$ is given by $\left( \frac{-x}{a}, \frac{1}{a} \right)$. The affine group is an example of a non-abelian and non-unimodular group since the left and right Haar measures are given by
    $$
    d\mu_\ell(x,a) = \frac{dx \,da}{a^2},\qquad d\mu_r(x, a) = \frac{dx \,da}{a},
    $$
    respectively and so in particular, $L^p_r(\aff) = L^p\big(\R \times \R^+, \frac{dx \,da}{a}\big)$.
    
    We will have use for dilates of sets in the affine group and so motivated by the group operation specified above, we define the scaling function $\Gamma_R : \aff \to \aff$ with parameter $R > 0$ as
    $$
    \Gamma_R(x, a) = (Rx, a^R),\qquad \Gamma_R^{-1}(x,a) = \left(\frac{x}{R}, a^{1/R}\right),\qquad  R\Omega = \big\{ \Gamma_R(x, a), (x, a) \in \Omega \big\}.
    $$
    This scaling relation is natural for the right Haar measure in the sense that it is the only one which is the identity for $R = 1$ and for which $\dmr(\Gamma_R(x,a)) = C\cdot \dmr(x,a)$ for some non-zero constant $C$. In particular, the constant $C$ has the value $R^2$ which can be verified directly and consequently,
    $$
    \dmr(\Gamma_R(x,a)) = R^2\dmr(x,a),\qquad \mu_r(R\Omega) = R^2\mu_r(R\Omega).
    $$
    For technical reasons, we will need to discuss convergence of sequences in $\aff$ as well as neighborhoods. To that end, we define the following distance function
    $$
    d_r^\aff \big( (x,a), (y,b) \big) = |x-y| + \Big|\ln\frac{a}{b}\Big|.
    $$
    It is chosen mostly for convenience but has the nice property that the distance from $(x,a)$ and $(0,1) = 0_\aff$ is given by $|x| + |\ln(a)|$ which is the sum of the horizontal and vertical distance when integrating with respect to the right Haar measure $\frac{dx\,da}{a}$.
    
    Using this distance function, we define the following type of balls in $\aff$:
    $$
    B_r^\aff\big( (x,a), \delta \big) = \big\{ (y,b) \in \aff : d_r^\aff\big((x,a), (y,b)\big) < \delta \big\}.
    $$
    Additional properties of the affine group relevant to quantum harmonic analysis are discussed in \cite{Berge2022}.
    
    \subsubsection{Abstract harmonic analysis}
    The approach for the Weyl-Heisenberg and affine groups described above can be generalized to the locally compact setting. Here we let $\sigma$ denote a square integrable unitary representation of a locally compact group $G$ and write $\mathcal{H}$ for the underlying Hilbert space so that $\sigma : G \to \mathcal{U}(\mathcal{H})$. This view is more closely connected to quantum mechanics and representation theory than time-frequency and time-scale analysis but many of the arguments work in the same way. Partly due to this connection to physics, $G$ is referred to as \emph{phase space}.
    
    \subsection{Motivating examples of square integrable representations}\label{sec:motivating_examples}
	Since the main contribution of this paper is setting up quantum harmonic analysis on general locally compact groups, we present two motivating examples on which the results apply and hint at their  generalizations to locally compact groups.
	\subsubsection{Shearlet group}
	The shearlet group represents an attempt to extend the wavelet transform to two-dimensional inputs. The dilations and one-dimensional translations of the wavelet transform are here replaced by asymmetric dilations, shears and two-dimensional translations using the \textit{parabolic scaling matrix} $A_a$ and the \textit{shear matrix} $S_s$, given by
	$$
	A_a = \begin{pmatrix}
	    a & 0\\
	    0 & \sqrt{a}
	\end{pmatrix},\qquad S_s = \begin{pmatrix}
	    1 & s\\
	    0 & 1
	\end{pmatrix}.
	$$
	The associated square integrable unitary representation can be written as
	$$
	\pi(a,s,x) \psi(t) = T_xD_{S_s A_a} \psi(t) = a^{-3/4} \psi\big(A_a^{-1} S_s^{-1} (t-x)\big)
	$$
	and it induces a group operation on the \emph{shearlet group} $\mathbb{S} = (\R^+ \times \R \times \R^2, \cdot_\mathbb{S})$,
	$$
	(a,s, x) \cdot_{\mathbb{S}} (a', s', x') = (aa', s+s'\sqrt{a}, x + S_s A_a x').
	$$
	The left and right Haar measures associated to the shearlet group can be computed to be
	$$
	\dml(a,s,x) = \frac{da\,ds\,dx}{a^3},\qquad \dmr(a,s,x) = \frac{da\, ds\, dx}{a}.
	$$
	For more on the shearlet group as well references for the statements above, see \cite{guo2006, DAHLKE2008, Kutyniok2012, Dahlke2009}. The shearlet group and associated transform has been generalized to higher dimensions which is also based on a square integrable representation, see \cite{Dahlke2009_multidimensional}.
    
    \subsubsection{Similitude group}
    The perhaps most straight-forward generalization of the wavelet transform to two-dimensional signals comes in the form of what is sometimes referred to as the \emph{two-dimensional wavelet transform} which is induced by the unitary representation
    $$
    \pi(a, x, \theta)\psi(t) = a^{-1}\psi\Big( \tau_{-\theta}\Big(\frac{t-x}{a}\Big) \Big)
    $$
    of the \emph{similitude group} $\mathrm{SIM}(2) = (\R^+ \times \R^2 \rtimes \mathrm{SO}(2), \cdot_{\mathrm{SIM}(2)})$ where $\tau_\theta \in \mathrm{SO}(2)$ is a rotation which acts as 
    $$
    \tau_\theta(x,y) = (x\cos\theta - y\sin\theta, x\sin\theta + y\cos\theta).
    $$
    The similitude group $\mathrm{SIM}(2)$ is equipped with the group operation 
    $$
    (a,x, \tau_\theta) \cdot_{\mathrm{SIM}(2)} (a', x', \tau_{\theta'}) = (aa', b+a\tau_\theta x', \tau_{\theta+\theta'}).
    $$
    Just as for shearlet group, the left and right Haar measures of $\mathrm{SIM}(2)$ dependend on the dilation parameter and are given by
    $$
	\dml(a,x,\theta) = \frac{da\,dx\,d\theta}{a^3},\qquad \dmr(a,x,\theta) = \frac{da\, dx\, d\theta}{a}.
    $$
    More on role of the similitude group in the two-dimensional wavelet transform can be found in \cite{Sagiv2006, Ali2000, Antoine1996, Dasgupta2022}. Similitude groups can be seen as a specific case of the affine group on $\R^d$ which generalizes the affine group by replacing dilations with multiplications by matrices in $\mathrm{GL}(\R^d)$ and normalized by the determinant. This also corresponds to a square integrable representation \cite{Kutyniok2012}.

    \subsubsection{Affine Poincaré group}
    Another approach to two-dimensional wavelet transforms comes in the form of the affine Poincaré group $\mathcal{P}_{\textrm{aff}}$ which consists of translations, zooming and hyperbolic rotations. More specifically, the group law takes the form
    $$
    (b,a,\vartheta) \cdot_{\mathcal{P}_\aff} (b', a', \vartheta') = (b+a\Lambda_\vartheta b', aa', \vartheta + \vartheta'),\qquad \Lambda_\vartheta = \begin{pmatrix}
        \cosh \vartheta & \sinh \vartheta\\
        \sinh \vartheta & \cosh \vartheta
    \end{pmatrix}
    $$
    Since we have the same sort of zooming and rotational system as in the shearlet and similitude groups, the affine Poincaré group has the same left and right Haar measures given by
    $$
	\dml(b,a,\vartheta) = \frac{db\,da\,d\vartheta}{a^3},\qquad \dmr(b,a,\vartheta) = \frac{db\, da\, d\vartheta}{a}.
    $$
    The natural square integrable representation of $\mathcal{P}_\aff$ is given by
    $$
    \pi(b,a,\vartheta)\psi(t) = \frac{1}{a}\psi\left( \frac{1}{a} \Lambda_\vartheta(t-b) \right)
    $$
    and it can be decomposed as the direct sum of four irreducible representations on orthogonal subsets of $L^2(\R^2)$.
    
    An introduction to the affine Poincaré group can be found in \cite[Sec. 7.4]{Antoine2004} and some recent work extending concepts from time-frequency analysis to $\mathcal{P}_\aff$ is available in \cite{Dasgupta2022_poincare}. The group also has applications in physics in the context of Minkowski spacetime \cite[Sec. 16.2.4]{Ali2014}.
    
	\section{Operator convolutions}\label{sec:convolutions}
	In this section we introduce the three types of convolutions we deal with: function-function, function-operator and operator-operator and prove some elementary properties and bounds. All of the definitions generalize those in \cite{Werner1984, Luef2018, Berge2022} and the proofs are similar with the exception of Proposition \ref{prop:function_operator_convolution_trace_class}. The function-function convolutions are standard but we write them down to be clear about the right Haar measure convention.
	\begin{definition}
		For $f, g \in L^1_r(G)$, the \emph{convolution} $f *_G g$ is defined as
		$$
		f *_{G} g(x) = \int_{G} f(y) g(x y^{-1}) \dmr(y).
		$$
	\end{definition}
	The following standard estimate follows by Young's inequality.
	\begin{proposition}\label{prop:function_convolution_L1_estimate}
		Let $f \in L^1_r(G)$ and $g \in L^p_r(G)$ for $1 \leq p \leq \infty$. Then
		$$
		\Vert f *_G g \Vert_{L^p_r(G)} \leq \Vert f \Vert_{L^1_r(G)} \Vert g \Vert_{L^p_r(G)}.
		$$
	\end{proposition}
	
	\subsection{Function-operator convolutions}
	Inspired by the notion of a shift for operators of the form
	$$
	\alpha_x(S) = \sigma(x)^* S \sigma(x)
	$$
	which moves a function in phase space by $x$, applies $S$ and then moves it back by $x^{-1}$, we have the following definition for function-operator convolutions.
	\begin{definition}
		Let $f \in L^1_r(G)$ and $S \in \mathcal{S}^1$, then the \emph{convolution} $f \star_G S$ is defined as the operator on $\mathcal{H}$ given by
		$$
		f \star_{G} S = \int_{G} f(x) \sigma(x)^* S \sigma(x) \dmr(x).
		$$
		This operator acts weakly in the way described in Section \ref{sec:vector_integration}, i.e. as
    	\begin{align*}
    		\big\langle f \star_G S \psi, \phi \big\rangle = \int_G f(x) \langle \sigma(x)^*S \sigma(x)\psi, \phi \rangle \dmr(x)
    	\end{align*}
    	and we define $S \star_G f = f \star_G S$.
	\end{definition}
	
	\begin{remark}
	    In an upcoming paper by the author in collaboration with Feichtinger and Luef, the function-operator convolution defined above is realized as a special case of measure-operator convolutions and then the description of  the weak action becomes a theorem, not a definition.
	\end{remark}
	The following boundedness property of function-operator convolutions is an important result which will be used extensively. It is somewhat analogous to the $p = 1$ case of Proposition \ref{prop:function_convolution_L1_estimate} and the corresponding statement for the $p > 1$ range is proved in Section \ref{sec:convol_mapping_properties}.
	\begin{proposition}\label{prop:function_operator_convolution_trace_class}
		Let $f \in L^1_r(G)$ and $S \in \mathcal{S}^1$. Then we have
		$$
		\Vert f \star_G S \Vert_{\mathcal{S}^1} \leq \Vert f \Vert_{L^1_r(G)} \Vert S \Vert_{\mathcal{S}^1}.
		$$
	\end{proposition}
	\begin{proof}
	    We control the trace-class norm of $f \star_G S$ by bounding $|\langle f \star_G S, T\rangle|$ for $T \in B(\mathcal{H})$ with $\Vert T \Vert_{B(\mathcal{H})} = 1$ as
	    \begin{align*}
    	     |\langle f \star_G S, T\rangle| &= |\tr\big((f \star_G S) T^*\big)|\\
    	     &= \left| \sum_n \big\langle (f \star_G S) e_n, T e_n\big\rangle \right|\\
    	     &\leq \sum_n \int_G |f(x)| |\langle \sigma(x)^* S \sigma(x)e_n, T e_n \rangle| \dmr(x)\\
    	     &= \int_G |f(x)| \sum_n |\langle T^* \sigma(x)^* S \sigma(x) e_n, e_n \rangle| \dmr(x)
	    \end{align*}
	    where we used in the last step Tonelli's Theorem. For each $x$, the above sum can be bounded by $\Vert S \Vert_{\mathcal{S}^1}$ using \cite[Thm. 18.11]{conway2000a}. Hence
	    $$
	    |\langle f \star_G S, T\rangle| \leq \int_G |f(x)| \Vert S \Vert_{\mathcal{S}^1} \dmr(x) = \Vert f \Vert_{L^1_r(G)} \Vert S \Vert_{\mathcal{S}^1}
	    $$
	    as desired.
	\end{proof}
	In the same way that the integral over a function-function convolution can be decoupled, the trace of a function-operator convolution may be written as a product in the following way.
	\begin{proposition}\label{prop:trace_of_func_op_conv}
	Let $f \in L^1_r(G)$ and $S \in \mathcal{S}^1$. Then 
	$$
	\tr(f \star_G S) = \tr(S)\int_G f(x) \dmr(x).
	$$
	\end{proposition}
	\begin{proof}
    We compute
    \begin{align*}
        \tr(f \star_G S) &= \sum_n \langle (f \star S) e_n, e_n \rangle\\
        &= \sum_n \int_G f(x)\langle (\sigma(x)^* S \sigma(x))e_n, e_n \rangle\dmr(x).
    \end{align*}
    By Tonelli's Theorem we have that
    \begin{align*}
        \sum_n \int_G |f(x)\langle \sigma(x)^* S \sigma(x)e_n, e_n \rangle| \dmr(x) &= \int_G \sum_n |f(x) \langle S \sigma(x)e_n, \sigma(x) e_n \rangle| \dmr(x)\\
        &\leq \int_G |f(x)| \dmr(x) \sum_n |\langle S \sigma(x)e_n, \sigma(x)e_n\rangle|
    \end{align*}
    where the first factor is finite by the integrability of $f$. The finiteness of the second factor follows by \cite[Prop. 18.9]{conway2000a}. We can now finish the computation as
    \begin{align*}
        \tr(f \star_G S) &= \int_G f(x) \sum_n \langle S \sigma(x) e_n, \sigma(x)e_n \rangle \dmr(x)\\
        &= \tr(S) \int_G f(x)\dmr(x).
    \end{align*}
	\end{proof}
    Lastly we show that function-operator convolutions preserve positivity.	
	\begin{lemma}\label{lemma:pos_pos_func_op_conv}
	    If $f \in L^1_r(G)$ is non-negative and $S \in \mathcal{S}^1$ is positive, then so is $f \star_G S$.
	\end{lemma}
	\begin{proof}
    	We verify this directly as
    	\begin{align*}
    		\left\langle (f \star_G S) \phi, \phi\right\rangle &= \int_G f(x) \langle \sigma(x)^* S \sigma(x) \phi, \phi\rangle \dmr(x)\\
    		&= \int_G f(x) \langle S \sigma(x) \phi, \sigma(x) \phi \rangle \dmr(x) \geq 0.
    	\end{align*}
	\end{proof}
	
    \subsection{Operator-operator convolutions}
    A central theme in quantum harmonic analysis is that when replacing functions by operators, integrals should be replaced by traces. This motivates the following definition of operator-operator convolutions.
	\begin{definition}
	    Let $T \in \mathcal{S}^1$ and $S \in B(\mathcal{H})$, then the \emph{convolution} $T \star_G S$ is the function on $G$ given by
		$$
		T \star_G S (x) = \tr(T \sigma(x)^* S \sigma(x)).
		$$
	\end{definition}
    The following lemma is an example of an operator-operator convolution.
	\begin{lemma}\label{lemma:rank_one_convolved_with_operator}
		For $\psi, \phi \in \mathcal{H}$ and $S \in B(\mathcal{H})$ we have
		$$
		(\psi \otimes \phi) \star_G S(x) = \big\langle S \sigma(x) \psi, \sigma(x)\phi \big\rangle.
		$$
	\end{lemma}
	\begin{proof}
		We compute
		\begin{align*}
			(\psi \otimes \phi) \star_G S(x) &= \tr((\psi \otimes \phi) \sigma(x)^* S \sigma(x))\\
			&= \sum_n \Big\langle \langle \sigma(x)^* S \sigma(x) e_n, \phi\rangle \psi, e_n\Big\rangle\\
			&= \langle \psi, \sigma(x)^* S^* \sigma(x) \phi \rangle\\
			&= \langle S \sigma(x)\psi, \sigma(x)\phi\rangle.
		\end{align*}
	\end{proof}
	\begin{lemma}\label{lemma:sum_op_conv_rank_one}
	    Let $S\in \mathcal{S}^1$ and $\{ \xi_n \}_n$ be an orthonormal basis of $\mathcal{H}$, then
	    $$
	    \sum_n (\xi_n \otimes \xi_n) \star_G S(x) = \tr(S).
	    $$
	\end{lemma}
	\begin{proof}
	    This follows directly from Lemma  \ref{lemma:rank_one_convolved_with_operator} and the fact that $\{ \sigma(x)\xi_n \}_n$ is an orthonormal basis.
	\end{proof}
	We also have the following estimate which follows by standard properties of trace norms, see \cite[Thm. 18.11 (g)]{conway2000a} for a proof.
	\begin{lemma}\label{lemma:operator_convolution_infinity_estimate}
		Let $T \in \mathcal{S}^1$ and $S \in B(\mathcal{H})$. Then
		$$
		\Vert T \star_G S\Vert_{L^\infty(G)} \leq \Vert T \Vert_{\mathcal{S}^1} \Vert S \Vert_{B(\mathcal{H})}.
		$$
	\end{lemma}
	As to be expected, operator-operator convolutions preserve positivity in the same way as usual function-function convolutions and function-operator convolutions in Lemma \ref{lemma:pos_pos_func_op_conv}.
	\begin{lemma}\label{lemma:positive_operator_convolution_is_positive}
	    Let $T \in \mathcal{S}^1$ and $S \in B(\mathcal{H})$ both be positive. Then
	    $$
	    T \star_G S(x) \geq 0\qquad\text{for all}~~ x \in G.
	    $$

	\end{lemma}
	\begin{proof}
	    We expand $T$ in its singular value decomposition and compute the trace with the same basis to find
	    \begin{align*}
	        T \star_G S(x) = \tr(T \sigma(x)^* S \sigma(x)) &= \sum_n \left\langle \sum_m \lambda_m (e_m \otimes e_m)(\sigma(x)^* S \sigma(x) e_n), e_n \right\rangle\\
	        &= \sum_{n,m} \lambda_m \langle (\sigma(x)^* S \sigma(x)) e_n, e_m \rangle \langle e_m, e_n \rangle\\
	        &= \sum_n \lambda_n \langle S \sigma(x)e_n, \sigma(x) e_n\rangle \geq 0,
	    \end{align*}
	    where we used that the eigenvalues of a positive operator are non-negative.
	\end{proof}
	Lastly, we show that all the convolutions introduced in this section are associative in an appropriate manner.
	\begin{proposition}\label{prop:compatibility_relations}
		Let $f, g \in L^1_r(G),\, T \in \mathcal{S}^1$, and let $S$ be a bounded operator on $\mathcal{H}$. Then the following compatibility relations hold
		\begin{align*}
			(f \star_G T) \star_G S &= f *_G (T \star_G S),\\    
			f \star_G (g \star_G T) &= (f *_G g) \star_G T.
		\end{align*}
	\end{proposition}
	\begin{proof}
		We proceed by direct computation
		\begin{align*}
			f *_G (T \star_G S) &= \int_G f(y) \tr\big(T \sigma(x y^{-1})^* S \sigma(x y^{-1})\big) \dmr(y)\\
			&= \int_G f(y) \tr\big(\sigma(y)^*T \sigma(y) \sigma(x)^* S \sigma(x)\big) \dmr(y)\\
			&= \tr\left( \left(\int_G f(y) \sigma(y)^* T \sigma(y) \dmr(y)\right) \sigma(x)^* S \sigma(x)  \right)\\
			&= \tr\big( (f \star_G T) \sigma(x)^* S \sigma(x) \big)\\
			&= (f \star_G T) \star_G S.
		\end{align*}
		For the other equality, we have
		\begin{align*}
			(f *_G g) \star_G T &= \int_G \left(\int_G f(x) g(z x^{-1}) \dmr(x) \right) \sigma(z)^*T \sigma(z) \dmr(z)\\
			&=\int_G \int_G f(x) g(z x^{-1}) \sigma(z)^*T \sigma(z) \dmr(z) \dmr(x)
		\end{align*}
		and so applying the change of variables $y = z x^{-1}$, we find
		\begin{align*}
			(f *_G g) \star_G T &= \int_G \int_G f(x) g(y) \sigma(yx)^*T \sigma(yx) \dmr(y) \dmr(x)\\
			&=\int_G f(x) \sigma(x)^* \left(\int_G g(y)\sigma(y)^*T \sigma(y)\dmr(y) \right) \sigma(x) \dmr(x)\\
			&=\int_G f(x) \sigma(x)^* \left(g \star_G T \right) \sigma(x) \dmr(x)\\
			&=f \star_G (g \star_G T).
		\end{align*}
	\end{proof}
	We have not established all the mapping properties of function-operator and operator-operator convolutions between $L^p_r(G)$ and $\mathcal{S}^p$ spaces, since this requires some more preparation which is contained in the next section, in particular \ref{sec:convol_mapping_properties}. 

	\section{Admissibility of operators}\label{sec:admissibility}
	\subsection{Integrability of operator-operator convolutions}
	As seen in \cite{Berge2022}, to generalize some results on operator convolutions to non-unimodular groups, we need to introduce a notion of \emph{admissibility of operators}. The definition is motivated by the desire for $T \star_G S$ to be integrable. Before stating such a result, we recall the following classical theorem from \cite{Duflo1976} which we write out using the right Haar measure convention.
	\begin{theorem}[Duflo-Moore]\label{thm:duflo_moore}
		Let $(\sigma, \mathcal{H})$ be a square integrable, irreducible, unitary representation of a locally compact group $G$. Then there exists a unique, possibly unbounded, densly defined, positive, closed, self-adjoint operator $\mathcal{D}^{-1} : \mathcal{H} \to \mathcal{H}$ with densly defined inverse $\D$ such that:
		\begin{enumerate}[label=(\roman*)]
		    \item The admissble vectors $\psi \in \mathcal{H}$ are exactly those vectors in the domain of $\mathcal{D}^{-1}$.
		    \item For $\phi_1, \phi_2 \in \mathcal{H}$ and $\psi_1, \psi_2 \in \operatorname{Dom}(\D^{-1})$, the following orthogonality relation holds
    		\begin{align}\label{eq:duflo_ortho_right}
    			\int_G \langle \phi_1, \sigma(x)^* \psi_1 \rangle \overline{\langle \phi_2, \sigma(x)^* \psi_2 \rangle} \dmr(x) = \langle \phi_1, \phi_2\rangle \overline{\langle \D^{-1} \psi_1, \D^{-1} \psi_2\rangle}.
    		\end{align}
    		\item The following covariance relation holds
    		\begin{align}\label{eq:duflo_moore_covariance_relation}
    		    \sigma(x) \D \sigma(x)^* = \sqrt{\Delta_G(x)} \D.
    		\end{align}
		\end{enumerate}
	\end{theorem}
	We can lift this result to the operator setting by taking the singular value decompositions of two operators and apply the above result to the rank-one situations.
	\begin{theorem}\label{thm:workhorse}
		Let $S \in \mathcal{S}^1$ satisfy $\mathcal{D}S\mathcal{D} \in B(\mathcal{H})$. For any $T \in \mathcal{S}^1$ we have that $T \star_G \mathcal{D}S\mathcal{D} \in L^1_r(G)$ with
		$$
		\Vert T \star_G \mathcal{D}S\mathcal{D} \Vert_{L^1_r(G)} \leq \Vert T \Vert_{\mathcal{S}^1} \Vert S \Vert_{\mathcal{S}^1}
		$$
		and
		\begin{align}\label{eq:workhorse_equality}
		    \int_G T \star_G \D S \D(x) \dmr(x) = \tr(T)\tr(S).    
		\end{align}
	\end{theorem}
    	\begin{proof}
		The proof consists of two steps.\\
		\textbf{Step 1:}
		We consider first the case where $S = \psi \otimes \phi$ for $\psi, \phi \in \mathcal{H}$. Then $\D S \D = \D \psi \otimes \D \phi$ and since $\mathcal{D} S \mathcal{D}$ is bounded by assumption, it in particular holds that $\psi, \phi \in \operatorname{Dom}(\D)$. To see this, note that
		$$
		\D S \D (f) = \langle f, \D \phi \rangle \D \psi
		$$
		and so since this is bounded for all $f$, $\D \psi$ and $\D \phi$ are also bounded.
		
		We now compute
		\begin{align*}
			T \star_G \D S \D(x) &= \tr(T \sigma(x)^* (\D \psi \otimes \D \phi) \sigma(x))\\
			&= \sum_n \langle T \sigma(x)^* (\D \psi \otimes \D \phi) \sigma(x) e_n, e_n\rangle\\
			&= \sum_n \langle \sigma(x) e_n, \D \phi\rangle \langle T \sigma(x)^*D \psi, e_n \rangle\\
			&= \langle T \sigma(x)^* \D \psi, \sigma(x)^* \D \phi \rangle.
		\end{align*}
		Since $T \in \mathcal{S}^1$, we can expand it using its singular value decomposition $T = \sum_n t_n \xi_n \otimes \eta_n$ which allows us to write
		\begin{align*}
			T \star_G \D S \D(x) &= \sum_n t_n \langle (\xi_n \otimes \eta_n)(\sigma(x)^* \D \psi), \sigma(x)^* \D \phi \rangle\\
			&= \sum_n t_n \langle \sigma(x)^* \D \psi, \eta_n \rangle \langle \xi_n, \sigma(x)^* \D \phi \rangle.
		\end{align*}
		Now each term is of the form in the Duflo-Moore orthogonality relation \eqref{eq:duflo_ortho_right}. We can therefore proceed by integrating each term after bounding the result as
		\begin{align*}
			\int_G \big|\langle \sigma(&x)^* \D \psi, \eta_n \rangle \langle \xi_n, \sigma(x)^* \D \phi \rangle\big| \dmr(x)\\
			&\leq \left( \int_G |\langle \sigma(x)^* \D \psi, \eta_n \rangle|^2 \dmr(x)\right)^{1/2}\left( \int_G |\langle \xi_n, \sigma(x)^* \D \phi \rangle|^2 \dmr(x)\right)^{1/2}\\
			&= \Vert \eta_n \Vert \Vert \D^{-1} \D \psi\Vert \Vert \xi_n \Vert \Vert \D^{-1} \D \phi \Vert = \Vert \psi \Vert \Vert \phi \Vert.
		\end{align*}
		Since $T \in \mathcal{S}^1$, $(t_n)_n$ is summable and we can move the integral inside to deduce that
		$$
		\Vert T \star_G \mathcal{D}S\mathcal{D} \Vert_{L^1_r(G)} \leq \Vert T \Vert_{\mathcal{S}^1} \Vert \psi \Vert \Vert \phi \Vert.
		$$
		We can now establish \eqref{eq:workhorse_equality} by moving the integral inside the sum and using the Duflo-Moore orthogonality relation \eqref{eq:duflo_ortho_right} which yields
		$$
		\int_G T \star_G \mathcal{D} S \mathcal{D} (x) \dmr(x) = \sum_n t_n \langle\xi_n, \eta_n \rangle \langle \phi, \psi \rangle = \tr(T) \langle \phi, \psi \rangle.
		$$
		\textbf{Step 2:}
		We now move to considering $S$ as in the theorem. By compactness, we can consider its singular value decomposition of $S$ which is of the form
		$$
		S = \sum_n s_n \psi_n \otimes \phi_n.
		$$
		Hence we have
		\begin{align*}
			T \star_G \D S \D &= T \star_G \D \left( \sum_n s_n \psi_n \otimes \phi_n \right) \D\\
			&= \sum_n s_n T \star_G \D (\psi_n \otimes \phi_n) \D\\
			&= \sum_n s_n T \star_G (\D \psi_n \otimes \D \phi_n)
		\end{align*}
		where we are allowed to move the sum outside since the outer sum is uniformly convergent by Lemma \ref{lemma:operator_convolution_infinity_estimate}. We can also estimate the norm as
		\begin{align*}
			\Vert T \star_G \D S \D \Vert_{L^1_r(G)} &\leq \sum_n |s_n| \Vert T \star_G (\D \psi_n \otimes \D \phi_n) \Vert_{L^1_r(G)}\\
			&= \Vert T \Vert_{\mathcal{S}^1} \Vert S \Vert_{\mathcal{S}^1}.
		\end{align*}
		For \eqref{eq:workhorse_equality}, the same method used in the first step yields the desired conclusion.
	\end{proof}

	Because we typically integrate $T \star_G S$ in applications, this theorem is more useful when considering the operator $\D^{-1} S \D^{-1}$ which leads us to make the following definition.
	\begin{definition}
		Let $S \neq 0$ be a bounded operator on $\mathcal{H}$ that maps $\operatorname{Dom}(\D)$ into $\operatorname{Dom}(\D^{-1})$. We say that $S$ is \emph{admissible} if the composition $\D^{-1} S \D^{-1}$ is bounded on $\operatorname{Dom}(\D^{-1})$ and extends to a trace-class operator $\D^{-1} S \D^{-1} \in \mathcal{S}^1$.
	\end{definition}
	We can now restate Theorem \ref{thm:workhorse} using the above definition.
	\begin{corollary}\label{corollary:workhorse_alternative}
		Let $S$ be an admissible operator and $T \in \mathcal{S}^1$. Then $T \star_G S \in L^1_r(G)$ with
		$$
		\Vert T \star_G S \Vert_{L^1_r(G)} \leq \Vert T \Vert_{\mathcal{S}^1} \Vert \D^{-1} S \D^{-1} \Vert_{\mathcal{S}^1}
		$$
		and
		$$
		\int_G T \star_G S(x) \dmr(x) = \tr(T)\tr(\D^{-1} S \D^{-1}).
		$$
	\end{corollary}
	\begin{example}\label{example:recover_duflo_moore}
	    As hinted at by the proof of Theorem \ref{thm:workhorse}, we can recover the orthogonality relation \eqref{eq:duflo_ortho_right} by choosing $S$ and $T$ as rank-one operators. To make this explicit, choose
	    $$
	    T = \phi_1 \otimes \phi_2,\qquad S = \psi_2 \otimes \psi_1
	    $$
	    where $\psi_1, \psi_2$ are admissible vectors in the sense of Theorem \ref{thm:duflo_moore}, i.e. $\psi_1, \psi_2 \in \operatorname{Dom}(\D^{-1})$. Then $S$ is admissible since
	    $$
	    \Vert \D^{-1} S \D^{-1} \Vert_{\mathcal{S}^1} = \Vert \D^{-1}\psi_2 \otimes \D^{-1}\psi_1 \Vert_{\mathcal{S}^1} \leq \Vert \D^{-1}\psi_2 \Vert \Vert \D^{-1} \psi_1 \Vert < \infty.
	    $$
	    Corollary \ref{corollary:workhorse_alternative} combined with Lemma \ref{lemma:rank_one_convolved_with_operator} now allows us to compute $T \star_G S(x)$ as
	    \begin{align*}
	        \int_G T \star_G S(x)\dmr(x) &= \int_G \langle \phi_1, \sigma(x)^*\psi_1 \rangle \overline{\langle \phi_2, \sigma(x)^*\psi_2 \rangle} \dmr(x)\\
	        &= \tr(\phi_1 \otimes \phi_2) \tr(\D^{-1} \psi_2 \otimes \D^{-1} \psi_1) \\
	        &= \langle \phi_1, \phi_2 \rangle \overline{\langle \D^{-1} \psi_1, \D^{-1} \psi_2 \rangle}
	    \end{align*}
	    which is exactly \eqref{eq:duflo_ortho_right}.
	\end{example}
	
    By demanding stronger conditions on both $S$ and $T$ we can deduce the following corollary.
	\begin{corollary}\label{corollary:admissible_left_right_integrability}
		Let $S$ and $T$ be admissible trace-class operators on $\mathcal{H}$, then the convolution $T \star_G S$ satisfies $T \star_G S \in L^1_r(G) \cap L^1_\ell(G)$ and
		\begin{align*}
			\int_G T \star_G S(x) \dmr(x) = \tr(T) \tr(\D^{-1} S \D^{-1}),\\
			\int_G T \star_G S(x) \dml(x) = \tr(S) \tr(\D^{-1} T \D^{-1}).
		\end{align*}
	\end{corollary}
	\begin{proof}
	    The first equality and $T \star_G S \in L^1_r(G)$ is the statement of Corollary \ref{corollary:workhorse_alternative}. The two corresponding statements for the left Haar measure follow by making the change of variables $x \mapsto x^{-1}$ and using that $\sigma(x^{-1})^* = \sigma(x), \dmr(x^{-1}) = \dml(x)$ and $T \star_G S (x^{-1}) = S \star_G T (x)$.
	\end{proof}
	\subsection{Conditions for admissibility}
	In this section we go over some useful conditions for admissibility, all of which are generalizations of results in \cite{Berge2022}. The first result shows how admissible functions are related to admissible operators.
	\begin{proposition}\label{prop:rank_one_admissibility_criteria}
		A rank-one operator $S = \psi \otimes \phi$ for non-zero $\psi, \phi$ is an admissible operator if and only if $\psi, \phi \in \mathcal{H}$ are admissible functions.
	\end{proposition}
	\begin{proof}
		If $S = \psi \otimes \phi$ is admissible, then $\mathcal{D}^{-1}S\D^{-1}$ is trace-class and in particular bounded. Hence for $\xi \in \operatorname{Dom}(\D^{-1})$
		$$
		\Vert \D^{-1} S \D^{-1} \xi \Vert = |\langle \D^{-1} \xi, \phi \rangle| \Vert \D^{-1} \psi \Vert = |\langle \xi, \D^{-1}\phi \rangle| \Vert \D^{-1} \psi \Vert < C \Vert \xi \Vert
        $$
        which implies that both $\psi$ and $\phi$ are in $\operatorname{Dom}(\D^{-1})$. The converse was shown in Example \ref{example:recover_duflo_moore}.
	\end{proof}
	The next proposition characterizes positive admissible operators.
	\begin{proposition}\label{prop:char_of_admissible_positive_compact}
		Let $S$ be a non-zero positive compact operator with spectral decomposition
		$$
		S = \sum_n s_n \xi_n \otimes \xi_n.
		$$
		Then $S$ is admissible if and only if each $\xi_n$ is admissible and 
		$$
		\sum_n s_n \Vert \mathcal{D}^{-1} \xi_n \Vert^2 < \infty.
		$$
	\end{proposition}
	\begin{proof}
		We first do the case where $S$ is admissible. Let $\xi \in \mathcal{H}$ with $\Vert \xi \Vert = 1$. Then by Lemma \ref{lemma:rank_one_convolved_with_operator} and linearity,
		\begin{align*}
		    (\xi \otimes \xi) \star_G S(x) &= \sum_n s_n \big\langle (\xi_n \otimes \xi_n) \sigma(x) \xi, \sigma(x) \xi \big\rangle\\
		    &= \sum_n s_n |\langle \sigma(x) \xi, \xi_n \rangle|^2.
		\end{align*}
		By integrating the above and using the monotone convergence theorem we find
		\begin{align*}
            \int_G \xi \otimes \xi \star_G S(x) \dmr(x) &= \sum_n s_n \int_G |\langle \sigma(x) \xi, \xi_n \rangle|^2 \dmr(x)\\
            &= \sum_n s_n \Vert \xi \Vert^2 \Vert \D^{-1} \xi_n \Vert^2
		\end{align*}
		where we used the Duflo-Moore relation \eqref{eq:duflo_ortho_right} in the last step. This sum is finite since the integral can be bounded using Corollary \ref{corollary:workhorse_alternative}, the fact that $S$ is admissible and $\Vert \xi \Vert = 1$.
		
		For the other direction, assume that each $\xi_n$ is admissible and $\sum_n s_n \Vert \D^{-1} \xi_n \Vert^2 < \infty$. It is then clear that
		\begin{align}\label{eq:DSD_ideal_form}
		   	\sum_n s_n (\D^{-1} \xi_n) \otimes (\D^{-1} \xi_n)
		\end{align}
		is trace-class by an application of the triangle inequality. It is however not clear that the above is equal to $\D^{-1} S \D^{-1}$ and that $S$ maps $\operatorname{Dom}(\D)$ into $\operatorname{Dom}(\D^{-1})$ when the sum defining $S$ is infinite. For a fixed $\psi \in \mathcal{H}$, the partial sums
		$$
		(S\psi)_M = \sum_{n=1}^M s_n \langle \psi, \xi_n \rangle \xi_n
		$$
		converge to $S \psi$ as $M \to \infty$ by definition and moreover, $(S \psi)_M$ is admissible for each $M$.
		
		The corresponding sequence of partial sums $\D^{-1}(S \psi)_M$ also converges in $\mathcal{H}$ since
		\begin{align*}
		    \sum_n s_n |\langle \psi, \xi_n \rangle| \Vert \D^{-1} \xi_n \Vert &\leq \left(\sum_n |\langle \psi, \xi_n \rangle|^2 \right)^{1/2}\left( \sum_n s_n^2 \Vert \D^{-1}\xi_n \Vert^2 \right)^{1/2}\\
		    &\leq C(S) \Vert \psi \Vert \left(\sum_n s_n \Vert \D^{-1} \xi_n \Vert^2 \right)^{1/2}
		\end{align*}
		where $C(S)$ is some constant depending only on $S$. By Theorem \ref{thm:duflo_moore}, the Duflo-Moore operator $\D^{-1}$ is closed and hence $S\psi$ is admissible with
		\begin{align}\label{eq:admissible_xi_n_respects_duflo_moore}
		    \D^{-1} S \psi = \sum_n s_n \langle \psi, \xi_n \rangle \D^{-1} \xi_n.
		\end{align}
		Now to show that $\D^{-1} S \D^{-1}$ is bounded on $\operatorname{Dom}(\D^{-1})$, let $\phi \in \operatorname{Dom}(\D^{-1})$ and note that by \eqref{eq:admissible_xi_n_respects_duflo_moore},
		\begin{align*}
		    \D^{-1} S \D^{-1} \phi &= \sum_n s_n \langle \D^{-1} \phi, \xi_n \rangle \D^{-1} \xi_n\\
		    &= \sum_n s_n \langle \phi, \D^{-1} \xi_n \rangle \D^{-1} \xi_n\\
		    &= \left(\sum_n s_n (\D^{-1} \xi_n) \otimes (\D^{-1} \xi_n)\right)\phi
		\end{align*}
		which we recognize as \eqref{eq:DSD_ideal_form}. This equality can be extended from the dense subspace $\operatorname{Dom}(\D^{-1})$ to all of $\mathcal{H}$ by denseness since
		$$
		\Vert \D^{-1} S \D^{-1} \phi \Vert \leq \sum_n s_n \Vert \phi \Vert \Vert \D^{-1} \xi_n \Vert^2 \leq \Vert \phi \Vert \sum_n s_n \Vert \D^{-1} \xi_n \Vert^2 < \infty
		$$
		by Cauchy-Schwarz and assumption.
	\end{proof}
    As a corollary, we have the following converse of Corollary \ref{corollary:workhorse_alternative}.
	\begin{corollary}\label{corollary:operator_convolution_integrability_implies_admissibility}
		Let $T$ be a non-zero positive trace-class operator and let $S$ be a non-zero positive compact operator. If
		$$
		\int_G T \star_G S(x) \dmr(x) < \infty,
		$$
		then $S$ is admissible with
		\begin{align}\label{eq:positive_admissibility_constant_implicit}
		    \tr(\D^{-1} S \D^{-1}) = \frac{1}{\tr(T)} \int_G T \star_G S(x) \dmr(x).
		\end{align}
		In particular, if $S$ is a non-zero, positive trace-class operator, then $S$ is admissible if and only
		if $S \star_G S \in L^1_r(G)$.
	\end{corollary}
	\begin{proof}
		Since $T$ and $S$ both are positive trace-class operators, we can expand them in their singular value decompositions with $S = \sum_n s_n \xi_n \otimes \xi_n$ and move the sum outside the integral using Tonelli's theorem to obtain
		\begin{align*}
		    \int_G T \star_G S(x) \dmr(x) = \tr(T) \sum_n s_n \Vert \D^{-1} \xi_n \Vert^2 < \infty
		\end{align*}
		where we used the Duflo-Moore orthogonality relation \eqref{eq:duflo_ortho_right} to compute the resulting inner integral. By Proposition \ref{prop:char_of_admissible_positive_compact}, we conclude that $S$ is admissible. Lastly Corollary \ref{corollary:workhorse_alternative} yields the equality \eqref{eq:positive_admissibility_constant_implicit}.
	\end{proof}
	Admissibility of $S$ does not automatically imply admissibility of $f \star_G S$ for $f \in L^p_r(G)$ as illustrated by the following proposition.
	\begin{proposition}\label{prop:function_operator_convolution_of_admissible_is_admissible}
		Suppose $f \in L^1_\ell(G) \cap L^1_r(G)$ be a non-zero and non-negative function. If $S$ is a positive, admissible trace-class operator on $\mathcal{H}$, then so is $f \star_G S$ with
		\begin{align}\label{eq:f_conv_S_trace_equality}
		    \tr(\D^{-1}(f \star_G S)\D^{-1}) = \int_G f(x) \dml(x) \tr(\D^{-1} S\D^{-1}).
		\end{align}
	\end{proposition}
	\begin{proof}
        That $f \star_G S$ is trace-class follows from Proposition \ref{prop:function_operator_convolution_trace_class} while positivity follows from Lemma \ref{lemma:pos_pos_func_op_conv}. Equation \eqref{eq:f_conv_S_trace_equality} follow if we can show that for $T$ non-zero, positive and trace-class,
		\begin{align}\label{eq:admissibility_criteria_subgoal}
		    \int_G T \star_G (f \star_G S)(y) \dmr(y) = \tr(T) \int_G f(x)\dmr(x) \tr(\D^{-1}S\D^{-1})
		\end{align}
		by applying Corollary \ref{corollary:operator_convolution_integrability_implies_admissibility} to $f \star_G S$ which also yields admissibility of $f \star_G S$. To show \eqref{eq:admissibility_criteria_subgoal}, we note that
		\begin{align*}
		    T \star_G (f \star_G S)(y) &= \tr\left( T \sigma(y)^* \int_G f(x) \sigma(x)^* S \sigma(x) \dmr(x)  \sigma(y) \right)\\
		    &= \int_G f(x) \tr\big( \sigma(x y)^* S \sigma(x y)\big) \dmr(x)\\
		    &= \int_G f(x) T \star_G S(xy) \dmr(x).
		\end{align*}
		From here we can integrate $T \star_G (f \star_G S)$ to find
		\begin{align*}
		    \int_G T \star_G (f \star_G S)(y) \dmr(y) &= \int_G   \int_G f(x) T \star_G S(x) \dmr(x y) \dmr(y)\\
		    &=\int_G f(x) \int_G T \star_G S(xy) \dmr(y) \dmr(x)\\
		    &= \int_G f(x) \dml(x)\tr(T)\tr(\D^{-1} S \D^{-1})
		\end{align*}
		where we used the change of variables $z = xy$ and the relations between the left and right Haar measure from Section \ref{sec:haar_measure}.
	\end{proof}
	
	\subsection{Interpolated convolution mapping properties}\label{sec:convol_mapping_properties}
	With the machinery of admissible operators in place, we can establish the remaining promised mapping properties from Section \ref{sec:convolutions} and generalize the inequalities
	\begin{align*}
	    \Vert f \star_G S \Vert_{\mathcal{S}^1} &\leq \Vert f \Vert_{L^1_r(G)} \Vert S \Vert_{\mathcal{S}^1},\\
	    \Vert T \star_G S \Vert_{L^\infty(G)} &\leq \Vert T \Vert_{B(\mathcal{H})} \Vert S \Vert_{\mathcal{S}^1}
	\end{align*}
	from Proposition \ref{prop:function_operator_convolution_trace_class} and Lemma \ref{lemma:operator_convolution_infinity_estimate} to all $\mathcal{S}^p$ and $L^p$-spaces. These results are generalizations of \cite[Prop. 4.16, Lem. 4.17, Prop. 4.18]{Berge2022} which treat the affine case and \cite[Prop. 4.2]{Luef2018} in the Weyl-Heisenberg case.
	\begin{proposition}\label{prop:interpolation_func-func_op-op_1}
	    Let $1 \leq p \leq \infty$ and let $q$ be its conjugate exponent given by $\frac{1}{p} + \frac{1}{q} = 1$. If $S \in \mathcal{S}^p, T \in \mathcal{S}^q$ and $f \in L^1_r(G)$, then the following holds:
	    \begin{enumerate}[label=(\roman*)]
	        \item\label{item:youngs_inequality} $f \star_G S \in \mathcal{S}^p$ with $\Vert f \star_G S \Vert_{\mathcal{S}^p} \leq \Vert f \Vert_{L^1_r(G)} \Vert S \Vert_{\mathcal{S}^p}$.
	        \item $T \star_G S \in L^\infty(G)$ with $\Vert T \star_G S \Vert_{L^\infty(G)} \leq \Vert S \Vert_{\mathcal{S}^p} \Vert T \Vert_{\mathcal{S}^q}$.
	    \end{enumerate}
	\end{proposition}
	\begin{proof}
	    The first inequality follows from \cite[Prop. 1.2.2]{Hytnen2016} for $p < \infty$ while the $p = \infty$ case can be deduced by an elementary estimate on the weak action $\langle f \star_G S \psi, \phi \rangle$. Meanwhile the second inequality follows from \cite[Thm. 2.8]{Simon2005}.
	\end{proof}
	Item \ref{item:youngs_inequality} above should be seen as a version of Young's inequality for function-operator convolutions and shows that $\mathcal{S}^p$ is a Banach module over $L^1_r(G)$ via the mapping $(f, S) \mapsto f \star_G S$. We now turn our attention to the case where $f$ and operator-operator convolutions are in $L^p_r(G)$ for $p \neq 1$. For interpolation purposes, we will first need the following lemma. 
 
	\begin{lemma}\label{lemma:quantization_admissible_weak_definition}
		Let $S \in \mathcal{S}^1$ and $f \in L^\infty(G)$. Define the operator $f \star_G \D S \D$ weakly for $\psi, \phi \in \operatorname{Dom}(\D)$ by
		$$
		\langle f \star_G \D S \D \psi, \phi \rangle = \int_G f(x) \langle S\mathcal{D}\sigma(x)\psi, \D \sigma(x)\phi \rangle \dmr(x).
		$$
		Then $f \star_G \D S \D$ uniquely extends to a bounded linear operator on $\mathcal{H}$ satisfying
		$$
		\Vert f \star_G \D S \D \Vert_{B(\mathcal{H})} \leq \Vert f \Vert_{L^\infty(G)} \Vert S \Vert_{\mathcal{S}^1}.
		$$
		In particular, if $R$ is an admissible operator, then $f \star_G R \in B(\mathcal{H})$ with
		$$
		\Vert f \star_G R \Vert_{B(\mathcal{H})} \leq \Vert f \Vert_{L^\infty(G)} \Vert \D^{-1} R \D^{-1} \Vert_{\mathcal{S}^1}.
		$$
	\end{lemma}
	\begin{proof}
		Using equation \eqref{eq:duflo_moore_covariance_relation}, we can rewrite $\langle f \star_G \D S \D \psi, \phi \rangle$ as
		\begin{align*}
		    \langle f \star_G \D S \D \psi, \phi \rangle &= \int_G f(x) \langle S \sigma(x) \mathcal{D}\psi, \sigma(x) \D \phi \rangle \dml(x)\\
		    &= \int_G f(x^{-1}) \langle S \sigma(x)^* \mathcal{D}\psi, \sigma(x)^* \D \phi \rangle \dmr(x)\\
		    &= \int_G f(x^{-1})(S \star_G (\D \psi \otimes \D \phi))(x) \dmr(x).
		\end{align*}
		To bound this, we note that $\D \psi \otimes \D \phi$ is an admissible operator and so by an elementary estimate of the integral and the use of Corollary \ref{corollary:workhorse_alternative}, we deduce that
		\begin{align*}
		    |\langle f \star_G \D S \D \psi, \phi \rangle| &\leq \Vert f \Vert_{L^\infty(G)}\Vert S \Vert_{\mathcal{S}^1} \Vert \D^{-1} (\D \psi \otimes \D \phi) \D^{-1} \Vert_{\mathcal{S}^1}\\
		    &= \Vert f \Vert_{L^\infty(G)}\Vert S \Vert_{\mathcal{S}^1} \Vert \psi \Vert \Vert \phi \Vert.
		\end{align*}
		That $f \star_G \D S \D$ extends uniquely follows from the denseness of $\operatorname{Dom}(\D)$ in $\mathcal{H}$.
	\end{proof}
	\begin{proposition}\label{prop:interpolated_schatten_memberships}
		Let $1 \leq p \leq \infty$ and let $q$ be its conjugate exponent given by $\frac{1}{p} + \frac{1}{q} = 1$. Suppose $S \in \mathcal{S}^1$ is admissible, $T \in \mathcal{S}^p$ and $f \in L_r^p(G)$.
		\begin{enumerate}[label=(\roman*)]
			\item\label{item:interpolation_function_operator} $f \star_G S \in \mathcal{S}^p$ with $\Vert f \star_G S \Vert_{\mathcal{S}^p} \leq \Vert f \Vert_{L_r^p(G)} \Vert S \Vert_{\mathcal{S}^1}^{1/p} \Vert \D^{-1} S \D^{-1} \Vert_{\mathcal{S}^1}^{1/q}$.
			\item\label{item:interpolation_operator_operator} $T \star_G S \in L_r^p(G)$ with $\Vert T \star_G S \Vert_{L_r^p(G)} \leq \Vert T \Vert_{\mathcal{S}^p} \Vert S \Vert_{\mathcal{S}^1}^{1/q} \Vert \D^{-1} S \D^{-1} \Vert_{\mathcal{S}^1}^{1/p}$.
		\end{enumerate}
	\end{proposition}
	
	\begin{proof}
		Item \ref{item:interpolation_function_operator} follows by complex interpolation between the $p=1$ case from Proposition \ref{prop:function_operator_convolution_trace_class} and the $p = \infty$ case which follows from Lemma \ref{lemma:quantization_admissible_weak_definition}.
		
		Similarly, item \ref{item:interpolation_operator_operator} follows by interpolating between the $p=1$ case from Corollary \ref{corollary:workhorse_alternative} and the $p = \infty$ case of Lemma \ref{lemma:operator_convolution_infinity_estimate}.
	\end{proof}
	With the interpolation results established, we can prove a generalization of \cite[Thm. 4.7]{Luef2018}, previously also noted in \cite{Werner1984} and \cite{Bayer2015} in the Weyl-Heisenberg case.
	\begin{proposition}\label{prop:convolution_adjoint}
        Let $S \in \mathcal{S}^1$ be admissible and define the maps
        \begin{align*}
            &\mathcal{A}_S : L^p_r(G) \to \mathcal{S}^p, \quad f \mapsto f \star_G S,\\
            &\mathcal{B}_S : \mathcal{S}^p \to L^p_r(G),\quad T \mapsto T \star_G S
        \end{align*}
        for $1 \leq p \leq \infty$. Then both maps are bounded and the adjoint of $\mathcal{A}_S : L^p_r(G) \to \mathcal{S}^p$ is given by $(\mathcal{A}_S)^* = \mathcal{B}_S : \mathcal{S}^q \to L^q_r(G)$ where $\frac{1}{p} + \frac{1}{q} = 1$ while the adjoint of $\mathcal{B}_S : \mathcal{S}^p \to L^p_r(G)$ is given by $(\mathcal{B}_S)^* = \mathcal{A}_S : L^q_r(G) \to \mathcal{S}^q$.
	\end{proposition}
	
	\begin{proof}
	    Boundedness of the mappings follows from Proposition \ref{prop:interpolated_schatten_memberships}. We compute the adjoint of $\mathcal{A}_S$ by considering the duality brackets
	    $$
	    \big\langle (\mathcal{A}_S)^* T, f \big\rangle_{L^p_r(G), L^q_r(G)} = \big\langle T, \mathcal{A}_S f \big\rangle_{\mathcal{S}^p, \mathcal{S}^q}
	    $$
	    for $T \in \mathcal{S}^p$ and $f \in L^q_r(G)$. The statement in the other direction for $(\mathcal{B}_S)^*$ will then follow by the same argument.
	    
	    First, assume that $T \in \mathcal{S}^1$ and $f \in L^1_r(G)$. It then holds that
	    \begin{align*}
	        \big\langle T, \mathcal{A}_S f \big\rangle &= \tr\big(T \mathcal{A}_S f\big)\\
	        &= \sum_n \big\langle T f \star_G Se_n, e_n \big\rangle\\
	        &= \sum_n \int_G f(x) \big\langle T\sigma(x)^*S\sigma(x)e_n, e_n\big\rangle \dmr(x).
	    \end{align*}
	    To justify the use of Fubini's Theorem on the above, we use Tonelli's Theorem to note that
	    \begin{align*}
	        \sum_n \int_G \big|f(x) \big\langle T\sigma(x)^*S\sigma(x)e_n, e_n\big\rangle\big|\dmr(x) = \int_G |f(x)| \sum_n \big|\big\langle T\sigma(x)^*S\sigma(x)e_n, e_n\big\rangle\big| \dmr(x).
	    \end{align*}
	    The operator $T \sigma(x)^*S\sigma(x)$ is is trace-class and the Hilbert-Schmidt norm of $\sigma(x)^* S \sigma(x)$ is independent of $x$ and so the proof of \cite[Prop. 18.9]{conway2000a} yields that the sum is uniformly bounded. By the integrability of $f$, we deduce that the entire quantity is finite. We can hence apply Fubini's Theorem to obtain
	    \begin{align*}
	        \big\langle T, \mathcal{A}_S f \big\rangle_{\mathcal{S}^p, \mathcal{S}^q} &= \int_G f(x) \sum_n\big\langle T\sigma(x)^*S\sigma(x)e_n, e_n \big\rangle \dmr(x)\\
	        &=\int_G f(x) T \star S(x) \dmr(x) = \big\langle \mathcal{B}_S T, f \big\rangle_{L^p_r(G), L^q_r(G)}.
	    \end{align*}
	    By the density of $\mathcal{S}^1$ and $L^1_r(G)$, this result can be extended to all of $\mathcal{S}^p$ and $L^q_r(G)$ while for $p=\infty$ the result holds by duality.
	\end{proof}

	Lastly we show how we can weaken the conditions on Proposition \ref{prop:function_operator_convolution_trace_class} but still get that the function-operator convolution is a compact operator which generalizes part of \cite[Lem. 2.3]{Luef2021}. We remind the reader that $L^0(G)$ is the set of all $L^\infty(G)$ functions which vanish at infinity while $\mathcal{K}$ is the set of compact operators. 
	\begin{corollary}\label{corollary:looser_conditions_on_func-op_compactness}
	    Let $f : G \to \C$ and $S \in B(\mathcal{H})$ satisfy one of the following
	    \begin{enumerate}[label=(\roman*)]
	        \item\label{item:func-op-compactness_condition_L0S1} $f \in L^0(G)$ and $S \in \mathcal{S}^1$ is admissible,
	        \item\label{item:func-op-compactness_condition_L1K} $f \in L^1_r(G)$ and $S \in \mathcal{K}$.
	    \end{enumerate}
	    Then $f \star_G S \in \mathcal{K}$.
	\end{corollary}
	\begin{proof}
	    Both alternatives follow by approximating the less well behaved object and showing that the approximations are compact and converge to $f \star_G S$ in the operator norm. For \ref{item:func-op-compactness_condition_L0S1}, let $f_n = f \chi_{B(0, n)}$ so that $f_n$ has compact support and $\Vert f_n - f \Vert_{L^\infty(G)} \to 0$ as $n \to \infty$. Then $f_n \in L^1_r(G)$ which implies that $f_n \star_G S \in \mathcal{S}^1$ by Proposition \ref{prop:function_operator_convolution_trace_class} and so $f_n \star_G S$ is compact. Lemma \ref{lemma:quantization_admissible_weak_definition} then yields that 
	    \begin{align*}
	        \Vert f \star_G S - f_n \star_G S \Vert_{B(\mathcal{H})} \leq \Vert f-f_n \Vert_{L^\infty(G)} \Vert \D^{-1} S \D^{-1} \Vert_{\mathcal{S}^1} \to 0
	    \end{align*}
	    as $n \to \infty$, which yields the desired conclusion.
	    
	    For \ref{item:func-op-compactness_condition_L1K}, since $S$ is compact, it is the limit in operator norm of a sequence of finite rank operators $(S_n)_n \subset \mathcal{S}^1$. By Proposition \ref{prop:function_operator_convolution_trace_class}, $f \star_G S_n$ is compact. Thus by Proposition \ref{prop:interpolation_func-func_op-op_1} with $p = \infty$,
	    $$
	    \Vert f \star_G S - f \star_G S_n \Vert_{B(\mathcal{H})} \leq \Vert f \Vert_{L^1_r(G)} \Vert S - S_n \Vert_{B(\mathcal{H})} \to 0
	    $$
	    as $n \to \infty$ and so $f \star S$ is also the limit in operator norm of a sequence of compact operators, hence it is compact.
	\end{proof}
	
	\section{Applications}\label{sec:applications}
	
	\subsection{Cohen's class distributions}
	Cohen's class of time-frequency distributions have a clear generalization to locally compact groups via quantum harmonic analysis and their integrability will in this section be shown to be connected to admissibility of an associated operator. In \cite{Luef2019} it was shown that, with the Weyl-Heisenberg group as the underlying group, any Cohen's class distribution can be written as in the following definition using the Weyl transform from Section \ref{sec:wigner_prelim}.
	\begin{definition}
	    A bilinear map $Q : \mathcal{H} \times \mathcal{H} \to L^\infty(G)$ is said to belong to the \emph{Cohen's class} if $Q = Q_S$ for some $S \in B(\mathcal{H})$ where
	    \begin{align}\label{eq:cohen_definition}
	        Q_S(\psi, \phi)(x) = (\psi \otimes \phi) \star_G S(x) = \langle S \sigma(x)\psi, \sigma(x)\phi \rangle.
	    \end{align}
	    We write $Q_S(\psi, \psi) = Q_S(\psi)$.
	\end{definition}
	Cohen's class distributions should be thought of as generalizations of the spectrogram and scalogram as illustrated by the following example.
	\begin{example}\label{ex:cohen_rank_one}
	When $S$ is the rank-one operator $S = \xi \otimes \eta$,
	$$
	Q_S(\psi, \phi)(x) = \langle\xi, \sigma(x) \phi \rangle \overline{\langle \eta, \sigma(x)\psi \rangle}.
	$$
	In particular, when $S = \xi \otimes \xi$,
	$$
	Q_S(\psi)(x) = |\langle \xi, \sigma(x)\psi|^2 = |\langle \psi, \sigma(x)^* \xi\rangle|^2
	$$
	which reduces down to the spectrogram in the Weyl-Heisenberg case and the scalogram in the affine case after replacing $x$ with $x^{-1}$. Note also that by the linearity of the mapping $S \mapsto Q_S$, if $S$ had been a finite rank operator, $Q_S$ would have been a sum of functions of the above form.
	\end{example}
	\begin{example}
	Because of the formalism we have set up, we can easily compute the Cohen's class distribution corresponding to the operator $f \star_G S$ using Proposition \ref{prop:compatibility_relations} as
	$$
	Q_{f \star_G S}(\psi, \phi) = (\psi \otimes \phi) \star_G (f \star_G S) = f *_G ((\psi \otimes \phi) \star_G S) = f *_G Q_S(\psi, \phi).
	$$
	\end{example}

	From here we can deduce some elementary properties of $Q_S$, generalizing the results in \cite[Prop. 6.9]{Berge2022} and \cite[Prop. 7.2, 7.3, 7.5]{Luef2019}.
	\begin{proposition}\label{prop:cohen_properties}
	    Let $S \in B(\mathcal{H})$. Then for $\psi, \phi \in \mathcal{H}$ the following properties hold:
	    \begin{enumerate}[label=(\roman*)]
	        \item \label{prop:cohen_properties:Linfinity_bound} The function $Q_S(\psi, \phi)$ satisfies
	        $$
	        \Vert Q_S(\psi, \phi)\Vert_{L^\infty(G)} \leq \Vert S \Vert_{B(\mathcal{H})} \Vert \psi \Vert \Vert \phi \Vert.
	        $$
	        \item \label{prop:cohen_properties:L1_bound} If $S$ is admissible, then $Q_S(\psi, \phi) \in L^1_r(G)$ and
	        $$
	        \int_G Q_S(\psi, \phi)(x) \dmr(x) = \langle \psi, \phi \rangle \tr(\D^{-1} \mathcal{S} \D^{-1}).
	        $$
	        \item \label{prop:cohen_properties:L1_bound_alternative} If $S$ is trace-class and $\psi, \phi$ are admissible, then $Q_S(\psi, \phi) \in L^1_\ell(G)$ and
	        $$
	        \int_G Q_S(\psi, \phi)(x) \dml(x) = \langle \D^{-1}\psi, \D^{-1}\phi \rangle \tr(\mathcal{S}).
	        $$
	        \item \label{prop:cohen_properties:covariance} We have the covariance property
	        \begin{align}\label{eq:cohen_covariance}
	            Q_S(\sigma(x) \psi, \sigma(x) \phi)(y) = Q_S(\psi, \phi)(yx)
	        \end{align}
	        for all $x, y \in G$.
	        \item \label{prop:cohen_properties:properties_from_S} The function $Q_S(\psi)$ is (real-valued) positive for all $\psi \in \mathcal{H}$ if and only if $S$ is (self-adjoint) positive.
	    \end{enumerate}
	\end{proposition}
	\begin{proof}
	    Item \ref{prop:cohen_properties:Linfinity_bound} is a direct consequence of Lemma \ref{lemma:operator_convolution_infinity_estimate} or alternatively the Cauchy-Schwarz inequality while item \ref{prop:cohen_properties:L1_bound} follows from Corollary \ref{corollary:workhorse_alternative} and item \ref{prop:cohen_properties:L1_bound_alternative} follows by an argument similar to that in the proof of Corollary \ref{corollary:admissible_left_right_integrability}. Item \ref{prop:cohen_properties:covariance} follows by a straight-forward calculation. Item \ref{prop:cohen_properties:properties_from_S} is clear from the definition \eqref{eq:cohen_definition}.
	\end{proof}
	
	\begin{remark}
	    In \cite{Luef2019}, Cohen's class distributions $Q_S$ are said to have the \emph{correct total energy property} if they satisfy
	    $$
	    \int_{\R^{2d}} Q_S(\psi, \phi)(x, \omega) \, dx\, d\omega = \langle \psi, \phi \rangle.
	    $$
	    This corresponds to $S$ being admissible with admissibility constant $\tr(\D^{-1} S \D^{-1}) = 1$ by item \ref{prop:cohen_properties:L1_bound} above.
	\end{remark}
	
	It also turns out that under rather loose conditions, any bilinear map on $\mathcal{H} \times \mathcal{H}$ is a Cohen's class distribution as shown in the following proposition generalizing \cite[Thm. 4.5.1]{grochenig_book} in the Weyl-Heisenberg case and \cite[Prop. 6.11]{Berge2022} in the affine case.
	\begin{proposition}
	    Let $Q : \mathcal{H} \times \mathcal{H} \to L^\infty(G)$ be a bilinear map. Assume that for all $\psi, \phi \in \mathcal{H}$ we know that $Q(\psi, \phi)$ is a continuous function on $G$ that satisfies the covariance property \eqref{eq:cohen_covariance} and the estimate
	    $$
	    |Q(\psi, \phi)(0_G)| \leq C\Vert \psi \Vert \phi \Vert
	    $$
	    for some constant $C$. Then there exists a unique bounded operator $S \in B(\mathcal{H})$ such that $Q = Q_S$.
	\end{proposition}
	\begin{proof}
	    By the Riesz representation theorem, there exists a bounded operator $S$ such that
	    $$
	    \langle S \psi, \phi \rangle = Q(\psi, \phi)(0_G)
	    $$
	    and that $Q = Q_S$ now follows from the covariance relation \eqref{eq:cohen_covariance}.
	\end{proof}
	In \cite[Thm. 7.6]{Luef2019}, positive Cohen's class distributions with the correct total energy property were characterized as (possibly infinite) convex combinations of spectrograms in the Weyl-Heisenberg case. An analogous result holds in the locally compact setting.
	\begin{theorem}
		If $S \in \mathcal{S}^1$ is positive, then there exists an orthonormal basis $\{ \varphi_n \}_n$ and a sequence $\{ \lambda_n \}_n$ of non-negative numbers with $\sum_n \lambda_n = \tr(S)$ such that
		$$
		Q_S(\psi)(x) = \sum_n \lambda_n |\langle \psi, \sigma(x)^* \varphi_n\rangle|^2
		$$
		where the convergence in the sum is uniform for any $\psi \in \mathcal{H}$.
	\end{theorem}
	\begin{proof}
		Since $S$ is trace-class and positive, it can be expanded in its singular value decomposition
		$$
		S = \sum_n \lambda_n \varphi_n \otimes \varphi_n
		$$
		where $\sum_n \lambda_n = \tr(S)$ by a theorem due to Lidskii \cite{Simon2005}. This allows us to write
		\begin{align*}
			Q_S(\psi)(x) &= (\psi \otimes \psi) \star_G \sum_n \lambda_n \varphi_n \otimes \varphi_n\\
			&= \sum_n \lambda_n (\psi \otimes \psi) \star_G (\varphi_n \otimes \varphi_n)\\
			&= \sum_n \lambda_n |\langle \psi, \sigma(x)^* \varphi_n \rangle|^2.
		\end{align*}
        That the convergence is uniform follows by Lemma \ref{lemma:operator_convolution_infinity_estimate} applied to $(\psi \otimes \psi) \star_G (\varphi_n \otimes \varphi_n)$.
	\end{proof}
	The following proposition highlights a connection between the operator $f \star_G S$ and the Cohen's class distribution $Q_S$. It is a straight-forward generalization of the Weyl-Heisenberg situation described in \cite[Prop. 8.2]{Luef2019} and the affine case considered in \cite[Prop. 6.12]{Berge2022}.
	\begin{proposition}\label{prop:minmax_loc_op}
	    Let $S$ be a positive, compact operator on $\mathcal{H}$ and let $f \in L^1_r(G)$ be a non-negative function. Then $f \star_G S$ is a positive, compact operator. Denote by $(\lambda_n)_n$ its eigenvalues in non-decreasing order with associated orthogonal eigenvectors $( \phi_n )_n$. Then
	    $$
	    \lambda_n = \max_{\Vert \psi \Vert = 1}\left\{ \int_G f(x)Q_S(\psi)(x) \dmr(x) : \psi \perp \phi_k \text{ for }k = 1, \dots, n-1  \right\}.
	    $$
	\end{proposition}
	\begin{proof}
	    Positivity of $f \star_G S$ follows from Lemma \ref{lemma:pos_pos_func_op_conv} while compactness follows from Corollary \ref{corollary:looser_conditions_on_func-op_compactness} \ref{item:func-op-compactness_condition_L1K}. The eigenvalue equality can be seen as a consequence of Courant's minimax theorem \cite[Thm. 28.4]{lax2002} upon noting that
	    $$
	    \langle f\star_G S \psi, \psi \rangle =  \int_G f(x)Q_S(\psi)(x) \dmr(x)
	    $$
	    which can be seen as a consequence of \eqref{eq:cohen_definition}.
	\end{proof}
	Since we have an $L^\infty$ bound on $Q_S(\psi)$ from Proposition \ref{prop:cohen_properties}, we can formulate an uncertainty principle in the same way as \cite[Cor. 7.7]{Luef2019} and \cite[Prop. 3.3.1]{grochenig_book} does for the Weyl-Heisenberg situation which says that if much of the mass of $Q_S$ is concentrated in a subset $\Omega \subset G$, $\Omega$ cannot be too small.
	\begin{corollary}
		Let $S \in B(\mathcal{H})$ and $\Omega$ be a compact subset of $G$ such that
		\begin{align}\label{eq:uncertainty_assumption}
		    \int_\Omega |Q_S(\psi)(x)| \dmr(x) \geq (1-\varepsilon) \Vert S \Vert_{B(\mathcal{H})}
		\end{align}
		for some $\psi \in \mathcal{H}$ with $\Vert \psi \Vert = 1$. Then,
		$$
		\mu_r(\Omega) \geq 1-\varepsilon.
		$$
	\end{corollary}
	\begin{proof}
	    Using \eqref{eq:uncertainty_assumption}, an elementary integral estimate and property \ref{prop:cohen_properties:Linfinity_bound} of Proposition \ref{prop:cohen_properties}, we obtain
	    \begin{align*}
		    (1-\varepsilon) \Vert S \Vert_{B(\mathcal{H})} &\leq \int_\Omega |Q_S(\psi)(x)| \dmr(x)\\
		    &\leq \Vert Q_S(\psi)\Vert_{L^\infty(G)} \mu_r(\Omega)\\
		    &\leq \Vert S \Vert_{B(\mathcal{H})} \mu_r(\Omega)
		\end{align*}
		which yields the desired inequality upon dividing out $\Vert S \Vert_{B(\mathcal{H})}$.
	\end{proof}
	\begin{example}
	Consider the case where $\mathcal{H} = L^2(\R^2)$ and $G = \mathbb{S}$, the shearlet group. If we let $\mathcal{SH}_\varphi \psi$ denote the shearlet transform of $\psi \in L^2(\R^2)$ with respect to the normalized window $\varphi$, the above corollary applied to a subset $\Omega \subset \mathbb{S}$ reads as
	$$
	\int_\Omega |\mathcal{SH}_\varphi\psi(a,s,x)|^2\, \frac{da\,ds\,dx}{a^3} \geq 1-\varepsilon \implies \int_\Omega \frac{da\,ds\,dx}{a} \geq 1-\varepsilon
	$$
	where we used the result of Example \ref{ex:cohen_rank_one} and the left and right Haar measure relations from Section \ref{sec:haar_measure}.
	\end{example}
	
	\subsection{Mixed-state localization operators}\label{sec:mix_state_loc}
	The localization operators introduced in Section \ref{sec:loc_op} have an equivalent formulation using quantum harmonic analysis for which the generalization to locally compact groups is clear.
	\begin{definition}
	    Let $f \in L^1_r(G)$ and $\varphi_1, \varphi_2 \in \mathcal{H}$. We then define the \emph{localization operator} $A_f^{\varphi_1, \varphi_2}$ on $\mathcal{H}$ as the operator
	    $$
	    A_f^{\varphi_1, \varphi_2} = f \star_G (\varphi_1 \otimes \varphi_2).
	    $$
	    By Proposition \ref{prop:function_operator_convolution_trace_class}, all localization operators are trace-class operators.
	\end{definition}
	\begin{remark}
	    The rank-one case of the interpolation results in Section \ref{sec:convol_mapping_properties} can be found stated for localization operators in \cite[Chap. 13]{Wong2002} which notably investigates localization operators on locally compact groups.
	\end{remark}
	Borrowing some terminology from quantum mechanics, the operator $\varphi_1 \otimes \varphi_2 = \varphi \otimes \varphi$ describes a \emph{pure state} of a system while a positive trace-class operator which does not have rank one describes a \emph{mixed state} since it is the limit of a linear combination of pure states. This leads us to the following two definitions, discussed in more depth in \cite{Luef2019, Luef2019_acc} for the Weyl-Heisenberg case.
	\begin{definition}
	A positive trace-class operator that is admissible with $\tr(\D^{-1} S \D^{-1}) = 1$ is said to be a \emph{density operator}.
	\end{definition}
	\begin{definition}
	Let $S$ be a density operator and $f \in L^1_r(G)$. Then the \emph{mixed-state localization operator} corresponding to $S$ and $f$ is defined as $f \star_G S$.
	\end{definition}
	
	\begin{remark}
	    It is possible to view localization operators as induced by a broader class of functions and operators and investigate the properties of those localization operators as in \cite{Cordero2003}. For the purposes of this paper we restrict our attention to localization operators as specified by the definitions above.
	\end{remark}
	\begin{remark}
	    Mixed-state localization operators show up as the natural analogue of localization operators for the operator wavelet transform defined in Section \ref{sec:operator_wavelet_transforms} below where the density operator condition is equivalent to the operator wavelet transform being an isometry.
	\end{remark}
	We are especially interested in the case where $f = \chi_\Omega$ where $\Omega$ is a compact subset of $G$ and $S$ is a density operator in which case we refer to $\chi_\Omega \star_G S$ as the mixed-state localization operator corresponding to $S$ and $\Omega$. By the positivity and compactness of $\chi_\Omega \star_G S$, we can write the singular value decomposition as
	$$
	\chi_\Omega \star_G S = \sum_k \lambda_k^\Omega h_k^\Omega \otimes h_k^\Omega
	$$
	where the eigenvalues $(\lambda_k^\Omega)_k$ are ordered decreasingly.
	
	In \cite{Luef2019_acc}, we have the following result on the distributions of the eigenvalues of mixed-state localization operators in the Weyl-Heisenberg case, essentially saying that the number of eigenvalues of $\chi_\Omega \star_G S$ close to $1$ scales as the size of the compact subset $\Omega$ provided $S$ is a density operator. This was originally proved in the rank-one case in \cite{Feichtinger2001}.
	\begin{theorem}[{\cite[Thm. 4.4]{Luef2019_acc}}]\label{theorem:approx_identity_mixed_loc_luef}
	    Let $S$ be a density operator on $L^2(\R^d)$, let $\Omega \subset \R^{2d}$ be a compact domain and fix $\delta \in (0,1)$. If $\big\{ \lambda_k^{R \Omega} \}_k$ are the non-zero eigenvalues of $\chi_{R\Omega} \star_{\mathbb{H}^n} S$, then
	    $$
	    \frac{\#\big\{ k : \lambda_k^{R\Omega} > 1 - \delta \big\}}{|R\Omega|} \to 1\quad \text{as }R \to \infty.
	    $$
	\end{theorem}
	The proof of this theorem uses of approximate identities which show up as a consequence of the scaling $R\Omega$. In the locally compact setting we cannot consider dilations and therefore settle for proving the corresponding statement for the affine group.
	\begin{theorem}\label{theorem:density_operator_count_eigenvalues_affine}
	    Let $S$ be a density operator on $L^2(\R^+)$, let $\Omega \subset \aff$ be a compact domain and fix $\delta \in (0,1)$. If $\big\{ \lambda_k^{R \Omega} \big\}_k$ are the non-zero eigenvalues of $\chi_{R\Omega} \star_\aff S$, then
	    \begin{align}\label{eq:affine_eigenvalues_quotient}
            \frac{\# \big\{ k : \lambda_k^{R\Omega} > 1 - \delta \big\}}{\tr(S) \mu_r(R\Omega)} \to 1\quad \text{as }R \to \infty.
	    \end{align}
	\end{theorem}
	Note that the scaled set $R\Omega$ is defined as in Section \ref{sec:affine_group} using the scaling function
	$$
	\Gamma_R: \aff \to \aff,\qquad \Gamma_R(x,a) = (Rx, a^R), \qquad \Gamma_R^{-1}(x,a) = \Big(\frac{x}{R}, a^{1/R}\Big)
	$$
	where $R\Omega = \{ \Gamma_R(x,a), (x,a \in \Omega) \}$.
	
	Before starting the proof, we will need to establish some auxiliary lemmas, some of which we state in the locally compact setting for the sake of generality.
	\begin{lemma}\label{lemma:mixed_state_loc_eigenvalue bounds}
	    Let $S$ be a density operator and $\Omega \subset G$ a compact domain. Then the eigenvalues $\big\{ \lambda_k^\Omega \big\}_k$ of $\chi_\Omega \star_G S$ satisfy $0 \leq \lambda_k^\Omega \leq 1$.
	\end{lemma}
	\begin{proof}
	All of the eigenvalues are non-negative and real-valued since $\chi_\Omega \star_G S$ is a positive operator by Lemma \ref{lemma:pos_pos_func_op_conv}. For the upper limit, we have
	\begin{align*}
	    \lambda_k^\Omega &= \big\langle (\chi_\Omega \star_G S) h_k^\Omega, h_k^\Omega \big\rangle = \int_G \chi_\Omega(x) \big\langle \sigma(x)^* S \sigma(x) h_k^\Omega, h_k^\Omega \big\rangle \dmr(x)\\
	    &\leq \int_G Q_S(h_k^\Omega)(x)\dmr(x) = \big\langle h_k^\Omega, h_k^\Omega \big\rangle = 1 
	\end{align*}
	where we used Proposition \ref{prop:cohen_properties} \ref{prop:cohen_properties:L1_bound}.
	\end{proof}
	
	\begin{lemma}\label{lemma:density_S_nonneg}
		Let $S$ be a density operator, then the function $\tilde{S} = S \star_G S$ is non-negative, nonzero at $0_G$ and has total integral $1$ with respect to both the left and right Haar measure.
	\end{lemma}
	\begin{proof}
	    Non-negativity of $\tilde{S}$ follows from Lemma \ref{lemma:positive_operator_convolution_is_positive} and plugging in $T = S$ and $x = 0_G$ in its proof yields
		\begin{align*}
			\tilde{S}(0_G) = \sum_n \lambda_n^2 > 0
		\end{align*}
		where $(\lambda_n)_n$ are the eigenvalues of $S$. Meanwhile the last statement follows from Corollary \ref{corollary:admissible_left_right_integrability} and $S$ being a density operator.
	\end{proof}
	
	The following technical lemma on approximations of the identity is standard in the Weyl-Heisenberg case but requires some work in the affine case. Note that the interior in the formulation below refers to the ball defined in Section \ref{sec:affine_group} as
	$$
	B_r^\aff\big( (x,a), \delta \big) = \big\{ (y,b) \in \aff : d_r^\aff\big((x,a), (y,b)\big) < \delta \big\}
	$$
	where
	$$
	d_r^\aff \big( (x,a), (y,b) \big) = |x-y| + \Big|\ln\frac{a}{b}\Big|.
	$$
    \begin{lemma}\label{lemma:approx_id_affine}
    Let $\phi \in L^1_r(\aff)$ be non-negative, have total integral $1$ with respect to both the left and right Haar measure and be nonzero at $(0,1)$. Moreover, let $\Omega \subset \aff$ be compact and $(y,b)$ be a point in the interior of $\Omega$, then 
    $$
    \lim_{R \to \infty} R^2\int_{\Omega} \phi\big((\Gamma_R(x,a))(\Gamma_R(y,b))^{-1}\big) \dmr(x,a) = 1.
    $$
    \end{lemma}
    \begin{proof}
        Using the change of variables $(z, c) = \Gamma_R(x, a)$ and $(w,u) = \Gamma_R(y,b)$ and the elementary observation $\Omega \subset \aff$, we see that
        $$
        \lim_{R \to \infty} R^2\int_{\Omega} \phi\big((\Gamma_R(x,a))(\Gamma_R(y,b))^{-1}\big) \dmr(x,a) \leq \int_{\aff} \phi\big((z,c) (w,u)^{-1}\big) \dmr(z,c) = 1.
        $$
        We devote the remainder of the proof to showing that the limit is greater than or equal to $1$ also. Choose $\delta$ so small that $B_r^\aff\big( (y,b), \delta \big) \subset \Omega$, then
        \begin{align*}
            &\lim_{R \to \infty} R^2\int_{\Omega} \phi\big((\Gamma_R(x,a))(\Gamma_R(y,b))^{-1}\big) \dmr(x,a)\\
            &\hspace{3cm}\geq \lim_{R \to \infty} R^2 \int_{B_r^\aff((y,b), \delta) } \phi\big((\Gamma_R(x,a))(\Gamma_R(y,b))^{-1}\big) \dmr(x, a)\\
            &\hspace{3cm}=\lim_{R \to \infty}\int_{\{ (z,c) : \Gamma_R^{-1}(z,c) \in B_r^\aff((y,b), \delta)\} } \phi\big((z,c) (w,u)^{-1}\big) \dmr(z, c).
        \end{align*}
        where we once again used the change of variables $(z,c) = \Gamma_R(x,a)$ and $(w,u) = \Gamma_R(y,b)$. We claim that for each point $(z,c)$, there exists an $R$ so large that $\Gamma_R^{-1}((z,c)) \in B_r^\aff((y,b), \delta)$ meaning that the region we are integrating over grows to all of $\aff$ as $R \to \infty$. The desired conclusion will then follow by the monotone convergence theorem. Indeed, a quick computation using that $\Gamma_R^{-1}(x,a) = \left(\frac{x}{R}, a^{1/R}\right)$ yields
        \begin{align*}
            \Gamma_R^{-1}(z,c) \in B_r^\aff\big(\Gamma_R^{-1}(w,u), \delta \big) &\iff d_r^\aff\big( \Gamma_R^{-1}(z,c), \Gamma_R^{-1}(w,u) \big) < \delta\\
            &\iff \left| \frac{z-w}{R} \right| + \left| \ln\frac{c^{1/R}}{u^{1/R}} \right| = \frac{1}{R}\left( |z-w| + \left| \ln\frac{c}{u}\right| \right) < \delta
        \end{align*}
        which clearly is true for large enough $R$. Hence 
        $$
        \lim_{R \to \infty} \int_{\{ (z,c) : \Gamma_R^{-1}(z,c) \in B_r^\aff((y,b), \delta)\} } \phi((z,c) (w,u)^{-1}) \dmr(z,c) = \int_\aff \phi(z w^{-1}) \dmr(z,c ) = 1
        $$
        which yields the desired conclusion.
    \end{proof}
    
	The following lemma generalizes \cite[Prop. 6.1]{Luef2019} and \cite[Lem. 4.2]{Luef2019_acc}.
	\begin{lemma}\label{lemma:mixed_state_loc_properties}
		Let $\Omega$ be a compact subset of $G$ and let $S \in \mathcal{S}^1$ be a positive operator. Then
	    \begin{enumerate}[label=(\roman*)]
			\item \label{item:sum_eigenvalues_mixed_state_loc} $\tr(\chi_\Omega \star_G S) = \tr(S)\mu_r(\Omega)$.
			\item $\tr((\chi_\Omega \star_G S)^2 ) = \int_\Omega \int_\Omega \tilde{S} (x y^{-1}) \dmr(x)\dmr(y)$.
			\item \label{item:sum_eigenvalues_measure}If $\big\{ \lambda_k^\Omega \big\}_k$ are the eigenvalues of $\chi_\Omega \star_G S$ counted with algebraic multiplicity, then
			$$
			\sum_k \lambda_k^\Omega = \tr(S)\mu_r(\Omega).
			$$
		\end{enumerate}
	\end{lemma}
	\begin{proof}
		\begin{enumerate}[label=(\roman*)]
			\item This follows from Proposition \ref{prop:trace_of_func_op_conv}.
			\item By an argument analogous to that in the proof of Proposition \ref{prop:compatibility_relations} and properties of the Haar measure, 
			$$
			S \star_G (\chi_\Omega \star_G S)(x) = \int_\Omega \tilde{S}(y x) \dmr(y).
			$$
			Hence,
        	\begin{align*}
        	    \tr((\chi_\Omega \star_G S)^2 ) &= ((\chi_\Omega \star_G S) \star_G (\chi_\Omega \star_G S))(0_G)\\
        	    &= \chi_\Omega * \left( \int_\Omega \tilde{S}(y \, \cdot) \dmr(y) \right)(0_G)\\
        	    &= \int_\Omega \int_\Omega \tilde{S} (x y^{-1}) \dmr(x)\dmr(y).
        	\end{align*}
			\item This follows from the fact that the sum of the eigenvalues counted with algebraic multiplicity is equal to the trace \cite{Simon2005} and Proposition \ref{prop:trace_of_func_op_conv}.
		\end{enumerate}
	\end{proof}
	From here we can follow the proof in \cite{Luef2019_acc}, which in turn follows the proof in \cite{Abreu2015}, by first generalizing \cite[Lem. 4.3]{Luef2019_acc} to the locally compact setting.
    \begin{lemma}\label{lemma:max_G_estimate}
    Let $S$ be a density operator, $\Omega \subset G$ be compact and fix $\delta \in (0,1)$. Then
    \begin{align*}
        &\Big| \#\big\{ k : \lambda_k^\Omega > 1-\delta \big\} - \tr(S) \mu_r(\Omega) \Big|\\
        &\hspace{15mm}\leq \max\left\{ \frac{1}{\delta}, \frac{1}{1-\delta} \right\} \left| \int_\Omega \int_\Omega \tilde{S}(x y^{-1}) \dmr(x) \dmr(y) - \tr(S)\mu_r(\Omega) \right|.
    \end{align*}
    \end{lemma}
	\begin{proof}
		Define the function
		\begin{align*}
		    G(t) = \begin{cases}
		        -t  \qquad &\text{if } 0 \leq t \leq 1-\delta,\\
		        1-t \qquad &\text{if } 1-\delta < t \leq 1
		    \end{cases}
		\end{align*}
		and consider the operator
		$$
		G(\chi_\Omega \star_G S) = \sum_k G(\lambda_k^\Omega) h_k^\Omega \otimes h_k^\Omega
		$$
		which is well defined since $0 \leq \lambda_k^\Omega \leq 1$ by Lemma \ref{lemma:mixed_state_loc_eigenvalue bounds}.
		
		The sequence $\big\{ G(\lambda_k^\Omega) \big\}_k$ is absolutely summable since $\big\{ \lambda_k^\Omega \big\}_k$ is summable and $|G(\lambda_k^\Omega)| = \lambda_k^\Omega$ for large enough $k$. It therefore follows that
		\begin{align*}
			\tr(G(\chi_\Omega \star_G S)) &= \sum_k G(\lambda_k^\Omega)= \#\big\{ k : \lambda_k^\Omega > 1-\delta \big\} - \tr(S)\mu_r(\Omega)
		\end{align*}
		using Lemma \ref{lemma:mixed_state_loc_properties} \ref{item:sum_eigenvalues_measure} and the above. Hence
		\begin{align*}
			\Big| \#\big\{ k : \lambda_k^\Omega > 1-\delta \big\} - \tr(S)\mu_r(\Omega) \Big| = |\tr\big(G(\chi_\Omega \star_G S)\big)| \leq \tr\big(|G|(\chi_\Omega \star_G S)\big).
		\end{align*}
		To bound this, we need an estimate on $|G(t)|$. We wish to decide the constant $A$ such that the polynomial $At(1-t)$ is always greater than $|G(t)|$. Note that it suffices to make sure that this is the case around $t = 1-\delta$ by the concavity of $At(1-t)$ and so we can set
		$$
		A = \sup_{t \in [0,1]} \frac{|G(t)|}{t(1-t)} = \max\left\{ \frac{1}{\delta}, \frac{1}{1-\delta} \right\} 
		$$
		yielding
	    $$
	    |G(t)| \leq \max\left\{ \frac{1}{\delta}, \frac{1}{1-\delta} \right\}(t - t^2).
	    $$
		With the bounds on $|G(t)|$ established, it follows that
		\begin{align*}
		    \Big| \#\big\{ k : \lambda_k^\Omega > 1-\delta \big\} - \tr(S)\mu_r(\Omega) \Big| &\leq \max\left\{ \frac{1}{\delta}, \frac{1}{1-\delta} \right\} \tr\big(\chi_\Omega \star_G S - (\chi_\Omega \star_G S)^2\big)
		\end{align*}
	    and we obtain the desired conclusion by applying Lemma \ref{lemma:mixed_state_loc_properties} to the right hand side.
	\end{proof}

	\begin{proof}[Proof of Theorem \ref{theorem:density_operator_count_eigenvalues_affine}]
	    By Lemma \ref{lemma:max_G_estimate}, we have the bound
	    \begin{align*}
	            \Big|\# \big\{ k : \, &\lambda_k^{R\Omega} > 1 - \delta \big\} - \tr(S)\mu_r(R\Omega)\Big|\\
                &\leq \max\left\{ \frac{1}{\delta}, \frac{1}{1-\delta} \right\} \left| \int_{R\Omega} \int_{R\Omega} \tilde{S}((x,a) (y,b)^{-1}) \dmr(x,a) \dmr(y,b) - \tr(S)\mu_r(R\Omega) \right|.
        \end{align*}
        Hence by dividing by $\tr(S)\mu_r(R\Omega)$ and setting $\phi = \tilde{S}/\tr(S)$ which has total integral 1 by Corollary \ref{corollary:workhorse_alternative} and the fact that $S$ is a density operator, we get
        \begin{align*}
            &\left|\frac{\# \big\{ k : \, \lambda_k^{R\Omega} > 1 - \delta \big\} }{\tr(S)\mu_r(R\Omega)} - 1\right|\\
            &\hspace{15mm}\leq \max\left\{ \frac{1}{\delta}, \frac{1}{1-\delta} \right\} \left|\frac{1}{\mu_r(R\Omega)} \int_{R\Omega} \int_{R\Omega} \phi((x,a) (y,b)^{-1}) \dmr(x,a) \dmr(y,b) - 1 \right|\\
            &\hspace{15mm} =\max\left\{ \frac{1}{\delta}, \frac{1}{1-\delta} \right\} \left| \int_{R\Omega} \frac{1}{\mu_r(R\Omega)} \left(\int_{R\Omega} \phi((x,a) (y,b)^{-1}) \dmr(x,a)  - 1\right)\dmr(y,b) \right|\\
            &\hspace{15mm} \leq\max\left\{ \frac{1}{\delta}, \frac{1}{1-\delta} \right\} \int_{R\Omega} \frac{1}{\mu_r(R\Omega)} \left| \int_{R\Omega} \phi((x,a) (y,b)^{-1}) \dmr(x,a)  - 1\right|\dmr(y, b).
        \end{align*}
        We now focus on showing that this approaches zero as $R \to \infty$. The change of variables $(x,a) = \Gamma_R(z,c)$ and $(y,b) = \Gamma_R(w, u)$ yields
        \begin{align*}
            &\int_{R\Omega} \frac{1}{\mu_r(R\Omega)} \left| \int_{R\Omega} \phi((x,a) (y,b)^{-1}) \dmr(x,a)  - 1\right|\dmr(y,b)\\
            &\hspace{2cm} =R^2\int_{\Omega}\frac{1}{\mu_r(R\Omega)} \left|R^2\int_{\Omega} \phi((\Gamma_R(z,c)) (\Gamma_R(w,u))^{-1}) \dmr(z,c) - 1\right|\dmr(w,u).
        \end{align*}
        From here we can move the $R \to \infty$ limit inside the integral by the compactness of $\Omega$ and the bound $\left|R^2\int_{\Omega} \phi((\Gamma_R(z,c)) (\Gamma_R(w,u))^{-1}) \dmr(z,c) - 1\right| \leq 2$. To conclude that the outer integral is zero, we need to make the integrand small even when making the change of variables back to $(x,a), (y,b)$ which requires that $\left|R^2\int_{\Omega} \phi((\Gamma_R(z,c)) (\Gamma_R(w,u))^{-1})\dmr(z,c) - 1\right|$ vanishes as $R \to \infty$. This is the contents of Lemma \ref{lemma:approx_id_affine} and so we are done.
	\end{proof}
	
	\begin{remark}
	    The formulation of Theorem \ref{theorem:density_operator_count_eigenvalues_affine} can be generalized to the locally compact setting as 
	    $$
        \frac{\# \big\{ k : \lambda_k^{\Omega_n} > 1 - \delta \big\}}{\tr(S) \mu_r(\Omega_n)} \to 1\quad \text{as }n \to \infty
	    $$
	    where $(\Omega_n)_n$ is an exhausting sequence of $G$. Unfortunately, the proof does not fully carry over to this formulation but as long as an approximate identity lemma analogous to Lemma \ref{lemma:approx_id_affine} can be proved the full statement should follow as well. We therefore expect generalizations of Theorem \ref{theorem:density_operator_count_eigenvalues_affine} to homogeneous groups or stratified Lie groups to hold.
	\end{remark}
	
	For an example of how density operators can be constructed, see the following proposition.
	\begin{proposition}
	    Let $\psi$ be an admissible vector such that $\Vert \mathcal{D} \psi \Vert = 1$. Then $S = \psi \otimes \psi$ is a density operator. Moreover, if $f \in L^1_\ell(G)$ is non-negative with $\Vert f \Vert_{L^1_\ell(G)} = 1$ and $S$ is a density operator, then $f \star_G S$ is also a density operator.
	\end{proposition}
	\begin{proof}
	    The first part is clear from the definition of density operators and the second part follows by Lemma \ref{lemma:pos_pos_func_op_conv} and Proposition \ref{prop:function_operator_convolution_of_admissible_is_admissible}.
	\end{proof}
    Density operators can also be constructed from linear combinations of admissible functions as in Proposition \ref{prop:char_of_admissible_positive_compact}.

    \begin{remark}
	    The construction carried out for the affine group in this section can just as easily be applied to the shearlet group by defining the dilation function
	    $$
	    \Gamma_R : \mathbb{S} \to \mathbb{S}, \qquad \Gamma_R(a,s,x) = \Big( a^R, \frac{s}{R}, \frac{x}{R} \Big),
	    $$
	    which has the property $\mu_r(R \Omega) = R^4 \mu_r(\Omega)$, and the distance function
	    $$
	    d_r^\mathcal{S}((a,s,x), (b,t,y)) = \left| \ln \frac{a}{b} \right| + |s-t| + |x-y|.
	    $$
	    The proof in Lemma \ref{lemma:approx_id_affine} then works in the same way with obvious modifications.
	\end{remark}
    
	\subsection{Covariant integral quantizations}
	Covariant integral quantizations on $G$ are maps $\Gamma_S$ given by
	$$
	\Gamma_S(f) = f \star_G S
	$$
	which include mixed-state localization operators as a special case. They have been studied by Gazeu and collaborators, motivated by applications in physics \cite{Gazeau2020, Gazeau2020_2, Gazeau2016}.

	The following result generalizes \cite[Prop. 3.2 (3)]{Werner1984} and \cite[Prop. 6.4]{Berge2022}.
	\begin{proposition}
		Let $T$ be a trace-class operator on $\mathcal{H}$. Then
		$$
		\Gamma_{\D T \D}(1) = 1 \star_G \D T \D = \tr(T) I_\mathcal{H}.
		$$
	\end{proposition}
	\begin{proof}
		Let $\psi, \phi \in \operatorname{Dom}(D)$, then using Lemma \ref{lemma:quantization_admissible_weak_definition},
		\begin{align*}
			\langle 1 \star_G \D T \D \psi, \phi \rangle &= \int_G \langle T \D \sigma(x) \psi, \D \sigma(x) \phi \rangle \dmr(x)\\
			&=\int_G (\psi \otimes \phi) \star_G (\D T \D)(x) \dmr(x)\\
			&= \tr(\psi \otimes \phi) \tr(T)\\
			&= \langle \psi, \phi \rangle \tr(T)
		\end{align*}
		by Lemma \ref{lemma:rank_one_convolved_with_operator} and Theorem \ref{thm:workhorse}.
	\end{proof}
	From the above proposition we see that when $\tr(T) = 1$, $\Gamma_{\D T \D}(1)$ is the identity operator which has the following resolution of identity as a consequence
	$$
	I_\mathcal{H} = \int_G \sigma(x)^* \D T \D \sigma(x) \dmr(x) \implies \langle \psi, \phi\rangle = \int_G \big\langle \D T \D \sigma(x)\psi, \sigma(x) \phi \big\rangle \dmr(x)
	$$
	It turns out that this property together with a few more uniquely determine linear maps from $L^\infty(G)$ to $B(\mathcal{H})$ which is the contents of the following theorem from \cite{Kiukas2006_2} which generalizes the formulations in \cite[Thm. 6.2]{Luef2019} and \cite[Thm. 6.5]{Berge2022}.
	\begin{theorem}
		Let $\Gamma : L^\infty(G) \to B(\mathcal{H})$ be a linear map satisfying
		\begin{enumerate}
			\item $\Gamma$ sends positive functions to positive operators,
			\item $\Gamma(1) = I_\mathcal{H}$,
			\item $\Gamma$ is continuous from the weak${}^*$ topology on $L^\infty(G)$ (as the dual space of $L^1_r(G)$) to the weak${}^*$ topology on $B(\mathcal{H})$,
			\item $\sigma(x)^* \Gamma(f) \sigma(x) = \Gamma(R_x^{-1} f)$.
		\end{enumerate}
    	Then there exists a unique positive trace-class operator $T$ with $\tr(T) = 1$ such that
    	$$
    	\Gamma(f) = f \star_G \D T \D.
    	$$
	\end{theorem}
	\begin{remark}
	    That $\Gamma(f) = f \star_G \D T \D$ satisfies all of these properties can be verified directly.
	\end{remark}
	
	\begin{proof}
	    The proof is a matter of translating our situation to one described in a remark in \cite{Kiukas2006} which references \cite{Kiukas2006_2} for the full proof. This follows by first considering the bijection $\Gamma \mapsto \Gamma_l$ where $\Gamma_l(f) = \Gamma(\check{f})$ and $\check{f}(x) = f(x^{-1})$. To translate the result back into the form $f \star_G \D T \D$ we can use equation \eqref{eq:duflo_moore_covariance_relation} from Theorem \ref{thm:duflo_moore}.
	\end{proof}
	By Proposition \ref{prop:convolution_adjoint}, the adjoint of $\Gamma_S = \mathcal{A}_S$ is given by the map $T \mapsto T \star_G S$ and so covariant integral quantizations can be seen as inducing operator-operator convolutions as well.
	
	\begin{remark}
	Werner refers to mappings of the kind discussed above as \emph{positive correspondence rules} in \cite{Werner1984}.    
	\end{remark}
	
	\subsection{Operator wavelet transforms}\label{sec:operator_wavelet_transforms}
	In \cite{skrettingland2020equivalent_norms}, the STFT of an element of $L^2(\R^d)$ with respect to a Hilbert-Schmidt operator is defined and a generalization of Moyal's identity is proved which mirrors the Duflo-Moore theorem. Later in \cite{McNulty2022}, this approach was generalized by letting the STFT also act on operators in what they call the \emph{operator STFT}. In this section we generalize both of these constructions to the locally compact setting.

    \subsubsection{Wavelet transform with operator window}\label{sec:operator_window_wavelet_transform}
    
    \begin{definition}
        Let $S$ be a bounded operator on $\mathcal{H}$ such that $S^*S$ is admissible and $\psi \in \mathcal{H}$. Then the \emph{wavelet transform with operator window} $\mathfrak{W}_S(\psi)$ of $\psi$ with window to $S$ is defined as the function 
	$$
	\mathfrak{W}_S(\psi)(x) = S \sigma(x) \psi \in \mathcal{H}
	$$
        for $x \in G$.
    \end{definition}
    \begin{remark}
        Our definition differs from that in \cite{skrettingland2020equivalent_norms, McNulty2022} in that we don't take the adjoint of $\sigma$ to stay in line with the right Haar measure convention of this paper. Note also that the condition that $S^*S$ is admissible corresponds to $S$ being a Hilbert-Schmidt operator in the unimodular case.
    \end{remark}
	The wavelet transform with operator window should be considered as an element of the Hilbert space $L^2_r(G, \mathcal{H})$ of equivalence classes of elements $\Psi : G \mapsto \mathcal{H}$ such that
	$$
	\Vert \Psi \Vert_{L^2_r(G, \mathcal{H})} = \left( \int_G \Vert \Psi(x) \Vert^2_{\mathcal{H}} \dmr(x) \right)^{1/2} < \infty
	$$
	where $\Psi \sim \Phi$ if $\Psi(x) = \Phi(x)$ in $\mathcal{H}$ for $\mu_r$-a.e. $x \in G$. The inner product in this space is given by
	$$
	\langle \Psi, \Phi \rangle_{L^2_r(G, \mathcal{H})} = \int_G \big\langle \Psi(x), \Phi(x) \big\rangle_\mathcal{H} \dmr(x).
	$$
	We are now ready to prove the following orthogonality relation, which in particular shows that the  transform indeed is an element of $L^2_r(G, \mathcal{H})$.
	\begin{proposition}\label{prop:orthogonality_rel_op}
	    Let $S_1, S_2 \in B(\mathcal{H})$ be such that $S_2^* S_1$ is admissible and $\psi_1, \psi_2 \in \mathcal{H}$. Then 
	    $$
        \big\langle \mathfrak{W}_{S_1}\psi_1, \mathfrak{W}_{S_2}\psi_2 \big\rangle_{L^2_r(G, \mathcal{H})} = \langle \psi_1, \psi_2 \rangle \langle S_1 \D^{-1}, S_2 \D^{-1} \rangle_{\mathcal{S}^2}.
        $$
        In particular, $\mathfrak{W}_S \psi \in L^2_r(G, \mathcal{H})$ for $S \in B(\mathcal{H})$ such that $S^*S$ is admissible and $\psi \in \mathcal{H}$.
    \end{proposition}
    \begin{proof}
        By writing
        \begin{align*}
            \big\langle \mathfrak{W}_{S_1}\psi_1(x), \mathfrak{W}_{S_2}\psi_2(x) \big\rangle &= \big\langle S_1 \sigma(x) \psi_1, S_2 \sigma(x) \psi_2 \big\rangle\\
            &= \big\langle \sigma(x)^* S_2^*S_1 \sigma(x) \psi_1, \psi_2 \big\rangle\\
            &= ((\psi_1 \otimes \psi_2) \star_G S_2^*S_1)(x)
        \end{align*}
        using Lemma \ref{lemma:rank_one_convolved_with_operator}, we see that this can be integrated using Corollary \ref{corollary:workhorse_alternative} to yield
        \begin{align*}
            \big\langle \mathfrak{W}_{S_1}\psi_1, \mathfrak{W}_{S_2}\psi_2 \big\rangle_{L^2_r(G, \mathcal{H})} &= \tr(\psi_1 \otimes \psi_2) \tr(\D^{-1}S_2^* S_1\D^{-1})\\
            &=\langle \psi_1, \psi_2 \rangle \tr(S_1\D^{-1} (S_2 \D^{-1})^*)\\
            &=\langle \psi_1, \psi_2 \rangle \langle S_1 \D^{-1}, S_2 \D^{-1} \rangle_{\mathcal{S}^2}
        \end{align*}
        as desired.
    \end{proof}
	\begin{remark}
    	Note that when $S_1 = \xi \otimes \phi_1, S_2 = \xi \otimes \phi_2$ and $\xi$ is normalized, we recover the familiar Duflo-Moore relation
    	$$
    	\big\langle \mathfrak{W}_{S_1}\psi_1, \mathfrak{W}_{S_2}\psi_2\big\rangle = \langle \psi_1, \psi_2 \rangle \overline{\langle \D^{-1} \phi_1, \D^{-1} \phi_2 \rangle}.
    	$$
        Moreover, when
    	$$
    	S_1 = \sum_{n=1}^N \xi \otimes \phi_n,\qquad S_2 = \sum_{m=1}^M \xi \otimes \eta_m
    	$$
    	for some normalized $\xi$, $S_2^*S_1$ is admissible precisely when $\phi_n$ and $\eta_m$ are admissible for each $n,m$ since
    	$$
    	S_2^*S_1 = \sum_{n=1}^N \sum_{m=1}^M \eta_m \otimes \phi_n
    	$$
    	by Proposition \ref{prop:rank_one_admissibility_criteria} characterizing admissible rank-one operators.
	\end{remark}

        \subsubsection{Operator wavelet transform}
        \begin{definition}
	Let $S$ be a bounded operator on $\mathcal{H}$ such that $S^*S$ is admissible and $T \in \mathcal{HS}$, then the \emph{operator wavelet transform} of $T$ with respect to $S$ is defined as
	$$
	\mathfrak{W}_S T(x) = S \sigma(x) T.
	$$
	\end{definition}
	Note that $\mathfrak{W}_S T$ maps elements of $G$ to operators on $\mathcal{H}$. In fact, we will show that $\mathfrak{W}_S T$ is an element of the Hilbert space $L^2_r(G, \mathcal{HS})$ provided $S$ satisfies the same admissibility criterion as for the wavelet transform with operator window. The inner product in this space is given by
	$$
	\langle A, B \rangle_{L^2_r(G, \mathcal{HS})} = \int_G \langle A(x), B(x) \rangle_{\mathcal{HS}} \dmr(x).
	$$
    As in the wavelet transform with operator window case, we still have a version of Moyal's identity which generalizes \cite[Prop. 3.4]{McNulty2022}.
	\begin{proposition}\label{prop:orthognality_rel_op_op}
	    Let $S_1, S_2$ be such that $S_2^* S_1$ is admissible and $T, R \in \mathcal{HS}$. Then
	    $$
        \big\langle \mathfrak{W}_{S_1} T, \mathfrak{W}_{S_2}R \big\rangle_{L^2_r(G, \mathcal{HS})} = \langle T, R\rangle_{\mathcal{HS}}\langle S_1 \D^{-1}, S_2 \D^{-1} \rangle_{\mathcal{HS}}.
	    $$
	    In particular, $\mathfrak{W}_S T \in L^2_r(G, \mathcal{HS})$ for $S \in B(\mathcal{H})$ such that $S^*S$ is admissible and $T \in \mathcal{HS}$.
	\end{proposition}
	\begin{proof}
	    We compute
	    \begin{align*}
	        \big\langle \mathfrak{W}_{S_1} T, \mathfrak{W}_{S_2}R \big\rangle_{L^2_r(G, \mathcal{HS})} &= \int_G \big\langle \mathfrak{W}_{S_1} T(x), \mathfrak{W}_{S_2}R(x) \big\rangle_{\mathcal{HS}} \dmr(x) \\
	        &= \int_G \tr\big( S_1 \sigma(x) T R^* \sigma(x)^* S_2^* \big) \dmr(x)\\
            &= \int_G ((TR^*) \star_G (S_2^* S_1)) (x) \dmr(x)\\
            &= \tr(T R^*) \tr(\D^{-1} S_2^* S_1 \D^{-1})\\
            &= \langle T, R\rangle_{\mathcal{HS}}\langle S_1 \D^{-1}, S_2 \D^{-1} \rangle_{\mathcal{HS}}
	    \end{align*}
	    where we in the second to last step used Corollary \ref{corollary:workhorse_alternative}.
	\end{proof}

        \begin{remark}
            More structure such as the Toeplitz operators discussed in \cite{McNulty2022} carry over to the locally compact setting with similar modifications as in the above but in the interest of brevity we leave this be.
        \end{remark}

	\subsection{A Berezin-Lieb inequality}
	The Berezin-Lieb inequality as investigated in \cite{Werner1984, Luef2018_berezin, Klauder2012} can be seen as a generalization of Corollary \ref{corollary:workhorse_alternative}. We present here a generalization to locally compact groups with the proof partially based on a proof for the Weyl-Heisenberg case in \cite{Dorfler2021}.
	\begin{theorem}\label{theorem:berezin-lieb}
	    Fix a positive $T \in \mathcal{S}^1$ and let $S \in \mathcal{S}^1$ be admissible. If $\Phi$ is a non-negative, convex and continuous function on a domain containing the spectrum of $\tr(S) T$ and the range of $T \star_G S$, then
	    $$
	    \int_G \Phi \circ (T \star_G S)(x) \dmr(x) \leq \tr\big(\Phi(\tr(S) T\big) \frac{\tr(\D^{-1} S \D^{-1})}{\tr(S)}
	    $$
	    where $\Phi(S)$ is defined by the functional calculus. Similarly, if $S \in \mathcal{S}^1$ is positive and admissible, $f \in L^\infty(G)$ is non-negative and $\Phi$ is a non-negative, convex and continuous function on a domain containing the spectrum of $f \star_G S$ and the range of $\tr(\D^{-1} S \D^{-1}) f$, then
	    $$
	    \tr(\Phi(f \star_G S)) \leq \frac{\tr(S)}{\tr(\D^{-1} S \D^{-1})} \int_G \Phi\big(\tr(\D^{-1} S \D^{-1}) f(x)\big)  \dmr(x).
	    $$
	\end{theorem}
	\begin{proof}
	    Expanding $T$ in its singular value decomposition and using Lemma \ref{lemma:rank_one_convolved_with_operator}, we obtain
	    \begin{align*}
	        (T \star_G S) (x) &= \sum_n \lambda_n ((\xi_n \otimes \xi_n) \star_G S)(x)\\
	        &= \sum_n \lambda_n \langle S \sigma(x) \xi_n, \sigma(x) \xi_n \rangle\\
	        &= \sum_n \tr(S)\lambda_n \frac{\langle S \sigma(x) \xi_n, \sigma(x) \xi_n \rangle}{\tr(S)}.
	    \end{align*}
	    Since $\{ \xi_n \}_n$ is an orthonormal basis, so is $\{ \sigma(x) \xi_n \}_n$ and so for each $x$ we can view the above as the integral over $(\tr(S)\lambda_n)_n$ with respect to a discrete probability measure. Applying Jensen's inequality, we find that
	    \begin{align*}
	        \Phi \circ (T \star_G S)(x) &= \Phi\left( \sum_n \tr(S)\lambda_n \frac{\langle S \sigma(x) \xi_n, \sigma(x) \xi_n \rangle}{\tr(S)} \right)\\
	        &\leq \sum_n \Phi(\tr(S)\lambda_n) \frac{1}{\tr(S)}\langle S \sigma(x) \xi_n, \sigma(x) \xi_n \rangle
	    \end{align*}
	    for each $x$. Integrating both sides yields
	    \begin{align*}
	        \int_G \Phi \circ (T \star_G S)(x) \dmr(x) &\leq \int_G \sum_n \Phi(\tr(S)\lambda_n) \frac{1}{\tr(S)}\langle S \sigma(x) \xi_n, \sigma(x) \xi_n \rangle\\
	        &= \sum_n \Phi(\tr(S) \lambda_n) \frac{1}{\tr(S)} \int_G (\xi_n \otimes \xi_n) \star_G S(x) \dmr(x)\\
	        &= \sum_n \Phi(\tr(S) \lambda_n) \frac{\tr(\D^{-1} S \D^{-1})}{\tr(S)}
	    \end{align*}
	    where we used Tonelli, Lemma \ref{lemma:rank_one_convolved_with_operator} and Corollary \ref{corollary:workhorse_alternative}.
	    
	    For the function statement, we use that $f \star_G S$ is a positive operator by Lemma \ref{lemma:pos_pos_func_op_conv} to expand $f \star_G S$ as $f \star_G S = \sum_n \lambda_n \xi_n \otimes \xi_n$ which yields
	    \begin{align*}
	        \Phi(f \star_G S) = \sum_n \Phi(\lambda_n) \xi_n \otimes \xi_n.
	    \end{align*}
	    Now before taking the trace of this, note that by the above,
	    $$
	    \big\langle \Phi(f \star_G S) \xi_n, \xi_n \big\rangle = \Phi(\lambda_n) = \Phi\big(\langle (f \star_G S) \xi_n, \xi_n \rangle\big)
	    $$
	    and so
	    \begin{align*}
	        \tr(\Phi(f \star_G S)) &= \sum_n \big\langle \Phi(f \star_G S)\xi_n, \xi_n \big\rangle\\
	        &= \sum_n \Phi\big(\langle (f \star_G S) \xi_n, \xi_n \rangle\big)\\
	        &= \sum_n \Phi\left( \int_G f(x) \langle S \sigma(x) \xi_n, \sigma(x)\xi_n \rangle \dmr(x)\right)\\
	        &= \sum_n \Phi\left( \int_G \tr(\D^{-1} S \D^{-1})f(x) \frac{\langle S \sigma(x) \xi_n, \sigma(x)\xi_n \rangle}{\tr(\D^{-1} S \D^{-1})} \dmr(x)\right).
	    \end{align*}
	    Now for each $n$, $\frac{\langle S \sigma(x) \xi_n, \sigma(x)\xi_n \rangle}{\tr(\D^{-1} S \D^{-1})} \dmr(x)$ can be viewed as a probability measure since the inner product $\langle S \sigma(x) \xi_n, \sigma(x)\xi_n \rangle$ integrates to $\tr(\D^{-1} S \D^{-1})$ by Lemma \ref{lemma:rank_one_convolved_with_operator} and Corollary \ref{corollary:workhorse_alternative}. We can therefore apply Jensen's inequality to get
	    \begin{align*}
	        \tr(\Phi(f \star_G S)) &\leq \sum_n \int_G \Phi\big(\tr(\D^{-1} S \D^{-1}) f(x)\big) \frac{\langle S \sigma(x) \xi_n, \sigma(x)\xi_n \rangle}{\tr(\D^{-1} S \D^{-1})} \dmr(x)\\
	        &= \int_G \Phi\big( \tr(\D^{-1} S \D^{-1}) f(x)\big) \frac{\sum_n\langle S \sigma(x) \xi_n, \sigma(x)\xi_n \rangle}{\tr(\D^{-1} S \D^{-1})} \dmr(x)\\
	        &= \frac{\tr(S)}{\tr(\D^{-1} S \D^{-1})} \int_G \Phi\big(\tr(\D^{-1} S \D^{-1}) f(x)\big) \dmr(x)
	    \end{align*}
	    where we used Tonelli to change the order of integration and summation.
	\end{proof}
	\begin{remark}
	A result similar to the rank-one case of the second inequality in the above theorem can be found in \cite[Thm. 14.11]{Wong2002}.
	\end{remark}
	
	\subsection{Wiener's Tauberian theorem}
	In this section we extend a theorem originally proved in \cite{Kiukas2012} which can also be found in \cite{Luef2018}. It can be viewed as a descendant to Wiener's classical Tauberian theorem \cite{Wiener1932} and is similar in spirit to the ideal formulation in \cite{Leptin1976}. To state it, we first need to introduce the concept of regular functions and operators.
	\begin{definition}
	A function $g \in L^p_r(G)$ is said to be $p$-regular if
	$$
	\overline{\operatorname{span}\big\{ g(\cdot\, x^{-1}) \big\}_{x \in G}} = L^p_r(G).
	$$
	Similarly, an operator $S \in \mathcal{S}^p$ is said to be $p$-regular if 
	$$
	\overline{\operatorname{span}\big\{ \sigma(x)^* S\sigma(x) \big\}_{x \in G}} = \mathcal{S}^p.
	$$
	\end{definition}
	We now state the theorem before a short discussion.
	\begin{theorem}
	    Assume that there exists an admissible operator $R \in \mathcal{S}^1$ such that $R \star_G R$ is regular, let $S \in \mathcal{S}^p$ be admissible, $1 \leq p \leq \infty$ and let $q$ be the conjugate exponent of $p$. Then the following are equivalent:
	    \begin{enumerate}
    	    \item \label{item:regularity_1} $S$ is $p$-regular,
    	    \item \label{item:regularity_2} If $f \in L^q_r(G)$ and $f \star_G S = 0$, then $f = 0$,
    	    \item \label{item:regularity_3} $\mathcal{S}^p \star_G S$ is dense in $L^p_r(G)$,
    	    \item \label{item:regularity_4} If $T \in \mathcal{S}^q$ and $T \star_G S = 0$, then $T = 0$,
    	    \item \label{item:regularity_5} $L^p_r(G) \star_G S$ is dense in $\mathcal{S}^p$,
    	    \item \label{item:regularity_6} $S \star_G S$ is $p$-regular,
    	    \item \label{item:regularity_7} For any regular $T \in \mathcal{S}^1, T \star_G S$ is $p$-regular. 
    	\end{enumerate}
	\end{theorem}
	Since for $1 \leq p \leq p' \leq \infty$, we have the inclusion $\mathcal{S}^p \subset \mathcal{S}^{p'}$, it suffices to find an admissible 1-regular operator to establish existence of $p$-regular operators for any $p \geq 1$. In the Weyl-Heisenberg case, the operator $S = \varphi_0 \otimes \varphi_0$ where $\varphi_0$ is the standard Gaussian is a regular operator as is verified directly in \cite{Luef2018}. It is not as easy to find such an operator in the locally compact case as the closest we have is the indicator function on a compact neighborhood of the origin for which the proof method does not translate.
	
	Apart from the existence of a regular operator, the proof in the Weyl-Heisenberg case from \cite{Luef2018} carries over with minimal modifications to account for the use of the right Haar measure as opposed to the left Haar measure, the change of the underlying group and the requirement of admissibility. We show three implications in detail, the first of which is notable because it requires Proposition \ref{prop:convolution_adjoint} and motivates the requirement for $S$ to be admissible, and leave the remaining required modifications from the proof in \cite{Luef2018} to the reader.
	\begin{proof}
	    \eqref{item:regularity_2}$\iff$\eqref{item:regularity_3}: By Proposition \ref{prop:convolution_adjoint}, the mapping $\mathcal{A}_S : f \mapsto f \star S$ is adjoint to $\mathcal{B}_S : T \mapsto T \star_G S$ and by \cite[Thm. 4.12]{rudin1991functional}, denseness of the image of a mapping is equivalent to injectiveness of its adjoint.
	    
	    ~\\\eqref{item:regularity_4}$\iff$\eqref{item:regularity_5}: This follows by the same argument as the above but with the roles of $\mathcal{A}_S$ and $\mathcal{B}_S$ reversed.
	    
	    ~\\\eqref{item:regularity_2}$\implies$\eqref{item:regularity_4}: Assume that $T \star_G S$ = 0 for some $T \in \mathcal{S}^q$, then $A \star_G T \star_G S = 0$ for any $A \in \mathcal{S}^1$ which implies that $A \star T = 0$ for all $A\in \mathcal{S}^1$ by \eqref{item:regularity_2}. In particular,
	    $$
	    A \star_G T (0) = \tr(AT) = \langle A, T^*\rangle = 0\quad \text{ for all }A \in \mathcal{S}^1
	    $$
	    and hence $T = 0$ as desired.
	    
	    Again, the remainder of the implications follow with similar small modifications from \cite{Luef2018} which we leave to the interested reader.
	\end{proof}
	
	\section*{Acknowledgements}
	The author wishes to extend their gratitude to Eirik Berge and Stine Marie Berge for helpful discussions and to Franz Luef for reading drafts of the paper as well as guidance.
	

	\printbibliography
\end{document}